\documentclass[oneside,english]{amsart}

\usepackage[latin1]{inputenc} 
\usepackage{setspace} 
\usepackage{amssymb}
\usepackage[]{epsf,epsfig,amsmath,amssymb,amsfonts,latexsym}
\usepackage{enumerate}

\usepackage{tikz}
\usetikzlibrary{arrows,decorations.markings, calc, matrix, shapes,automata,fit,patterns,decorations.pathreplacing,calligraphy}
\usepackage{color}
\usepackage{cancel}
\usepackage{amsmath}
\usepackage[normalem]{ulem}
\usepackage{amsfonts}
\usepackage{bbm}
\usepackage{diagbox}
\usepackage{amsthm}

\usepackage[mathscr]{eucal}
\usepackage[active]{srcltx}
\usepackage{fancyhdr}
\usepackage{verbatim}
\usepackage{fullpage}
\usepackage{hyperref}
\usepackage{cleveref}
\usepackage{mathtools}
\allowdisplaybreaks

\def \NN{\mathbb N}
\def \ZZ{\mathbb Z}

\def \RR{\mathbb R}

\def \ag{A}
\def \FF{\mathcal F}

\newcommand{\symb}[1]{{#1}}

\newcommand{\define}[1]{\textbf{#1}}
	
\newcommand{\isdef}{:=}

\newcommand{\seed}{
	\clip (-0.5,-0.5) rectangle (+0.5,+0.5);
	\draw[black, fill =black!80] (-0.5,0.5)--(0,0)--(-0.5,-0.5)--cycle;
	\draw[black, fill=black!80] (-0.5,-0.5)--(0,0)--(+0.5,-0.5)--cycle;
	\draw[black, fill=black!80] (0.5,-0.5)--(0,0)--(+0.5,+0.5)--cycle;
	\draw[black, fill=black!50] (0.5,0.5)--(0,0)--(-0.5,+0.5)--cycle;
	\node at (0,0.35) {\color{white}\textbf{$\texttt{S}$}};
    \draw[black, fill=white] (-0.25,-0.25)rectangle(0.25,0.25);
	\node at (0,0) {\includegraphics[scale=0.2]{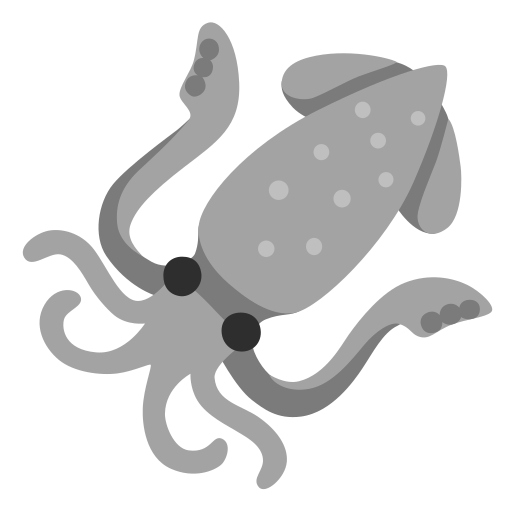}};
}
\newcommand{\seedL}{
	\clip (-0.5,-0.5) rectangle (+0.5,+0.5);
	\draw[black, fill = black!80] (-0.5,0.5)--(0,0)--(-0.5,-0.5)--cycle;
	\draw[black, fill=black!50] (-0.5,-0.5)--(0,0)--(+0.5,-0.5)--cycle;
	\draw[black, fill=black!50] (0.5,-0.5)--(0,0)--(+0.5,+0.5)--cycle;
	\draw[black, fill=black!50] (0.5,0.5)--(0,0)--(-0.5,+0.5)--cycle;
	\node at (0,0.35) {\color{white}\textbf{$\texttt{B}$}};
	\node at (0,-0.35) {\color{white}\textbf{$\texttt{S}$}};
	\node at (0.35,0) {\color{white}\textbf{$\texttt{C}$}};
    \draw[black, fill=white] (-0.25,-0.25)rectangle(0.25,0.25);
	\node at (0,0) {\includegraphics[scale=0.2]{octopus_gs.png}};
}

\newcommand{\seedR}{
	\clip (-0.5,-0.5) rectangle (+0.5,+0.5);
	\draw[black, fill = black!50] (-0.5,0.5)--(0,0)--(-0.5,-0.5)--cycle;
	\draw[black, fill=black!80] (-0.5,-0.5)--(0,0)--(+0.5,-0.5)--cycle;
	\draw[black, fill=black!80] (0.5,-0.5)--(0,0)--(+0.5,+0.5)--cycle;
	\draw[black, fill=black!50] (0.5,0.5)--(0,0)--(-0.5,+0.5)--cycle;
	\node at (-0.35,0) {\color{white}\textbf{$\texttt{C}$}};
	\node at (0,0.35) {\color{white}\textbf{$\texttt{C}$}};
}

\newcommand{\lefttile}{
	\clip (-0.5,-0.5) rectangle (+0.5,+0.5);
	\draw[black, fill=black!80] (-0.5,0.5)--(0,0)--(-0.5,-0.5)--cycle;
	\draw[black, fill=black!50] (-0.5,-0.5)--(0,0)--(+0.5,-0.5)--cycle;
	\draw[black, fill=white] (0.5,-0.5)--(0,0)--(+0.5,+0.5)--cycle;
	\draw[black, fill=black!50] (0.5,0.5)--(0,0)--(-0.5,+0.5)--cycle;
	\node at (0,0.35) {\color{white}\textbf{$\texttt{B}$}};
	\node at (0,-0.35) {\color{white}\textbf{$\texttt{B}$}};
	\draw[black, fill=white] (-0.25,-0.25)rectangle(0.25,0.25);
	\node at (0,0) {\includegraphics[scale=0.2]{octopus_gs.png}};
}

\newcommand{\midtile}{
	\clip (-0.5,-0.5) rectangle (+0.5,+0.5);
	\draw[black, fill=white] (-0.5,0.5)--(0,0)--(-0.5,-0.5)--cycle;
	\draw[black, fill=black!50] (-0.5,-0.5)--(0,0)--(+0.5,-0.5)--cycle;
	\draw[black, fill=black!50] (0.5,-0.5)--(0,0)--(+0.5,+0.5)--cycle;
	\draw[black, fill=black!50] (0.5,0.5)--(0,0)--(-0.5,+0.5)--cycle;
	\node at (0.35,0) {\color{white}\textbf{$\texttt{E}$}};
	\node at (0,0.35) {\color{white}\textbf{$\texttt{O}$}};
	\node at (0,-0.35) {\color{white}\textbf{$\texttt{C}$}};
}

\newcommand{\righttile}{
	\clip (-0.5,-0.5) rectangle (+0.5,+0.5);
	\draw[black, fill=black!50] (-0.5,0.5)--(0,0)--(-0.5,-0.5)--cycle;
	\draw[black, fill=black!80] (-0.5,-0.5)--(0,0)--(+0.5,-0.5)--cycle;
	\draw[black, fill=black!80] (0.5,-0.5)--(0,0)--(+0.5,+0.5)--cycle;
	\draw[black, fill=black!50] (0.5,0.5)--(0,0)--(-0.5,+0.5)--cycle;
	\node at (-0.35,0) {\color{white}\textbf{$\texttt{E}$}};
	\node at (0,0.35) {\color{white}\textbf{$\texttt{E}$}};
}

\newcommand{\origintile}{
	\clip (-0.5,-0.5) rectangle (+0.5,+0.5);
	\draw[black, fill=white] (-0.5,0.5)--(0,0)--(-0.5,-0.5)--cycle;
	\draw[black, fill=black!50] (-0.5,-0.5)--(0,0)--(+0.5,-0.5)--cycle;
	\draw[black, fill=white] (0.5,-0.5)--(0,0)--(+0.5,+0.5)--cycle;
	\draw[black, fill=black!15] (0.5,0.5)--(0,0)--(-0.5,+0.5)--cycle;
	\node at (0,-0.35) {\color{white}\textbf{$\texttt{O}$}};
	\node at (0,0.35) {\textbf{$(q_0,\sqcup)$}};
}

\newcommand{\blanktile}{
	\clip (-0.5,-0.5) rectangle (+0.5,+0.5);
	\draw[black, fill=white] (-0.5,0.5)--(0,0)--(-0.5,-0.5)--cycle;
	\draw[black, fill=black!50] (-0.5,-0.5)--(0,0)--(+0.5,-0.5)--cycle;
	\draw[black, fill=black!50] (0.5,-0.5)--(0,0)--(+0.5,+0.5)--cycle;
	\draw[black, fill=white] (0.5,0.5)--(0,0)--(-0.5,+0.5)--cycle;
	\node at (0.35,0) {\color{white}\textbf{$\texttt{E}$}};
	\node at (0,-0.35) {\color{white}\textbf{$\texttt{E}$}};
	\node at (0,0.35) {\textbf{$\sqcup$}};
}

\newcommand{\tiletransmit}[1]{
	\clip (-0.5,-0.5) rectangle (+0.5,+0.5);
	\draw[black, fill=white] (-0.5,0.5)--(0,0)--(-0.5,-0.5)--cycle;
	\draw[black, fill=white] (-0.5,-0.5)--(0,0)--(+0.5,-0.5)--cycle;
	\draw[black, fill=white] (0.5,-0.5)--(0,0)--(+0.5,+0.5)--cycle;
	\draw[black, fill=white] (0.5,0.5)--(0,0)--(-0.5,+0.5)--cycle;
	\node at (0,0.35) {\textbf{$#1$}};
	\node at (0,-0.35) {\textbf{$#1$}};
}
\newcommand{\tilestateup}[4]{
	\clip (-0.5,-0.5) rectangle (+0.5,+0.5);
	\draw[black, fill=white] (-0.5,0.5)--(0,0)--(-0.5,-0.5)--cycle;
	\draw[black, fill=black!15] (-0.5,-0.5)--(0,0)--(+0.5,-0.5)--cycle;
	\draw[black, fill=white] (0.5,-0.5)--(0,0)--(+0.5,+0.5)--cycle;
	\draw[black, fill=black!15] (0.5,0.5)--(0,0)--(-0.5,+0.5)--cycle;
	\node at (0,0.35) {\textbf{$(#3,#4)$}};
	\node at (0,-0.35) {\textbf{$(#1,#2)$}};
}

\newcommand{\tilestateupconditional}[5]{
	\clip (-0.5,-0.5) rectangle (+0.5,+0.5);
	\draw[black, fill=white] (-0.5,0.5)--(0,0)--(-0.5,-0.5)--cycle;
	\draw[black, fill=black!15] (-0.5,-0.5)--(0,0)--(+0.5,-0.5)--cycle;
	\draw[black, fill=white] (0.5,-0.5)--(0,0)--(+0.5,+0.5)--cycle;
	\draw[black, fill=black!15] (0.5,0.5)--(0,0)--(-0.5,+0.5)--cycle;
	\node at (0,0.35) {\textbf{{\small $(#3,#4)$}}};
	\node at (0,-0.35) {\textbf{{\small $(#1,#2)$}}};
	\draw[black, fill=white] (-0.15,-0.15)rectangle(0.15,0.15);
	\node at (0,0) {$#5$};
}

\newcommand{\tilestateleft}[4]{
	\clip (-0.5,-0.5) rectangle (+0.5,+0.5);
	\draw[black, fill=black!15] (-0.5,0.5)--(0,0)--(-0.5,-0.5)--cycle;
	\draw[black, fill=black!15] (-0.5,-0.5)--(0,0)--(+0.5,-0.5)--cycle;
	\draw[black, fill=white] (0.5,-0.5)--(0,0)--(+0.5,+0.5)--cycle;
	\draw[black, fill=white] (0.5,0.5)--(0,0)--(-0.5,+0.5)--cycle;
	\node at (0,0.35) {\textbf{$#4$}};
	\node at (-0.35,0) {\textbf{$\stackrel{#3}{\leftarrow}$}};
	\node at (0,-0.35) {\textbf{$(#1,#2)$}};	
}
\newcommand{\tilestateright}[4]{
	\clip (-0.5,-0.5) rectangle (+0.5,+0.5);
	\draw[black, fill=white] (-0.5,0.5)--(0,0)--(-0.5,-0.5)--cycle;
	\draw[black, fill=black!15] (-0.5,-0.5)--(0,0)--(+0.5,-0.5)--cycle;
	\draw[black, fill=black!15] (0.5,-0.5)--(0,0)--(+0.5,+0.5)--cycle;
	\draw[black, fill=white] (0.5,0.5)--(0,0)--(-0.5,+0.5)--cycle;
	\node at (0,0.35) {\textbf{$#4$}};
	\node at (0.35,0) {\textbf{$\stackrel{#3}{\rightarrow}$}};
	\node at (0,-0.35) {\textbf{$(#1,#2)$}};
}
\newcommand{\tilestateincomingleft}[2]{
	\clip (-0.5,-0.5) rectangle (+0.5,+0.5);
	\draw[black, fill=black!15] (-0.5,0.5)--(0,0)--(-0.5,-0.5)--cycle;
	\draw[black, fill=white] (-0.5,-0.5)--(0,0)--(+0.5,-0.5)--cycle;
	\draw[black, fill=white] (0.5,-0.5)--(0,0)--(+0.5,+0.5)--cycle;
	\draw[black, fill=black!15] (0.5,0.5)--(0,0)--(-0.5,+0.5)--cycle;
	\node at (0,0.35) {\textbf{$(#1,#2)$}};
	\node at (-0.35,0) {\textbf{$\stackrel{#1}{\rightarrow}$}};
	\node at (0,-0.35) {\textbf{$#2$}};
}
\newcommand{\tilestateincomingright}[2]{
	\clip (-0.5,-0.5) rectangle (+0.5,+0.5);
	\draw[black, fill=white] (-0.5,0.5)--(0,0)--(-0.5,-0.5)--cycle;
	\draw[black, fill=white] (-0.5,-0.5)--(0,0)--(+0.5,-0.5)--cycle;
	\draw[black, fill=black!15] (0.5,-0.5)--(0,0)--(+0.5,+0.5)--cycle;
	\draw[black, fill=black!15] (0.5,0.5)--(0,0)--(-0.5,+0.5)--cycle;
	\node at (0,0.35) {\textbf{$(#1,#2)$}};
	\node at (0.35,0) {\textbf{$\stackrel{#1}{\leftarrow}$}};
	\node at (0,-0.35) {\textbf{$#2$}};
}

\newcommand{\tilecommandleft}[1]{
	\clip (-0.5,-0.5) rectangle (+0.5,+0.5);
	\draw[black, fill=black!80] (-0.5,0.5)--(0,0)--(-0.5,-0.5)--cycle;
	\draw[black, fill=black!50] (-0.5,-0.5)--(0,0)--(+0.5,-0.5)--cycle;
	\draw[black, fill=white] (0.5,-0.5)--(0,0)--(+0.5,+0.5)--cycle;
	\draw[black, fill=black!50] (0.5,0.5)--(0,0)--(-0.5,+0.5)--cycle;
	\node at (0,0.35) {\color{white}\textbf{$\texttt{B}$}};
	\node at (0,-0.35) {\color{white}\textbf{$\texttt{B}$}};
	\draw[black, fill=white] (-0.25,-0.25)rectangle(0.25,0.25);
	\node at (0,0) {\includegraphics[scale=0.2]{octopus_gs.png}};
    \node at (0.35,0) {$#1$};
}

\newcommand{\tilecommandright}[1]{
	\clip (-0.5,-0.5) rectangle (+0.5,+0.5);
	\draw[black, fill=white] (-0.5,0.5)--(0,0)--(-0.5,-0.5)--cycle;
	\draw[black, fill=black!15] (-0.5,-0.5)--(0,0)--(+0.5,-0.5)--cycle;
	\draw[black, fill=white] (0.5,-0.5)--(0,0)--(+0.5,+0.5)--cycle;
	\draw[black, fill=black!15] (0.5,0.5)--(0,0)--(-0.5,+0.5)--cycle;
	\node at (0,0.35) {\textbf{$(q_{#1},\sqcup)$}};
	\node at (0,-0.35) {\textbf{$(q'_{#1},\sqcup)$}};
    \node at (-0.35,0) {\textbf{${#1}$}};
}

\newtheorem{theorem}{Theorem}[section]
\newtheorem{lemma}[theorem]{Lemma}

\newtheorem{proposition}[theorem]{Proposition}
\newtheorem{corollary}[theorem]{Corollary}
\newtheorem{definition}[theorem]{Definition}

\theoremstyle{remark}
\newtheorem{remark}[theorem]{Remark}

\newtheorem{question}[theorem]{Question}

\newcommand{\xConfig}[1]{%
	\begin{tikzpicture}[
		baseline=-\the\dimexpr\fontdimen22\textfont2\relax,ampersand replacement=\&]
		\matrix[
		matrix of math nodes,
		nodes={
			minimum size=1.4ex,text width=1.4ex,
			text height=1.4ex,inner sep=3pt,draw={gray!20},anchor=center
		}, row sep=1pt,column sep=1pt
		] (config) {#1};
		\node[draw,rectangle,help lines,gray!50, dashed,fit=(config),inner sep=-1pt] {};
	\end{tikzpicture}
}

\title{Soficity of free extensions of effective subshifts}

\author{
	Sebasti\'an Barbieri, Mathieu Sablik and Ville Salo
}

\newcommand{\Addresses}{{
		\bigskip

		\hskip-\parindent   S.~Barbieri, \textsc{Universidad de Santiago de Chile.}\par\nopagebreak
		\textit{E-mail address}: \texttt{sebastian.barbieri@usach.cl}
		
		\medskip
		
		\hskip-\parindent   M.~Sablik, \textsc{Universit\'e Paul Sabatier.}\par\nopagebreak
		\textit{E-mail address}: \texttt{mathieu.sablik@math.univ-toulouse.fr}
		
		\medskip
		
		\hskip-\parindent   V.~Salo, \textsc{University of Turku.}\par\nopagebreak
		\textit{E-mail address}: \texttt{vosalo@utu.fi}
}}

\begin{document}
	\large

\begin{abstract}
	  Let $G$ be a group and $H\leqslant G$ a subgroup. The free extension of an $H$-subshift $X$ to $G$ is the $G$-subshift $\widetilde{X}$ whose configurations are those for which the restriction to every coset of $H$ is a configuration from $X$. We study the case of $G = H \times K$ for infinite and finitely generated groups $H$ and $K$: on the one hand we show that if $K$ is nonamenable and $H$ has decidable word problem, then the free extension to $G$ of any $H$-subshift which is effectively closed is a sofic $G$-subshift. On the other hand we prove that if both $H$ and $K$ are amenable, there are always $H$-subshifts which are effectively closed by patterns whose free extension to $G$ is non-sofic. We also present a few applications in the form of a new simulation theorem and a new class of groups which admit strongly aperiodic SFTs.

	\medskip


	\smallskip
	
	\textit{Secondary:}
	37B05,  
	20F10.  
	
\end{abstract}

\maketitle


\section{Introduction}

Let $X$ be an effectively closed $\ZZ$-subshift, that is, a set of bi-infinite words which can be described by a recursively enumerable set of forbidden words. Consider the \define{trivial extension} of $X$ to $\ZZ^2$, that is, the $\ZZ^2$-subshift whose rows are elements of $X$ and which is constant on every column. The simulation theorem proven independently by Aubrun and Sablik~\cite{AubrunSablik2010} and Durand, Romaschenko and Shen~\cite{DurandRomashchenkoShen2010} shows that this trivial extension is a sofic $\ZZ^2$-subshift: it is the topological factor of a $\ZZ^2$-subshift of finite type (SFT).

If we look at the class of $\ZZ^2$-subshifts whose restriction to $\ZZ \times \{0\}$ coincides with a fixed $\ZZ$-subshift $X$, then the trivial extension is the most rigid element of this class: the choice of configuration from $X$ completely determines the $\ZZ^2$-configuration in the trivial extension. The least rigid member of this class is the \define{free extension}, where each coset $\ZZ\times \{n\}$ holds an arbitrary configuration from $X$. While the trivial extension is very useful in preserving some dynamical properties from $X$, others, such as the topological entropy, are preserved by free extensions~\cite{Barbieri2021_GGD,raymond2023shifts}.

It turns out that in general, the free extension of an effectively closed $\ZZ$-subshift is not sofic and thus the analogue of the simulation result for free extensions does not hold when passing from $\ZZ$ to $\ZZ^2$. In fact, it is an open question of Jeandel whether there exists a non-sofic effectively closed $\ZZ$-subshift whose free extension to $\ZZ^2$ is sofic. It is interesting to ask whether some analogue can still hold if we replace $\ZZ^2$ by a different algebraic structure.

The goal of this article is to study simulation results for free extensions in the larger context of finitely generated groups. More precisely, consider a group $G$ and $H \leqslant G$ a subgroup. The free extension to $G$ of an $H$-subshift $X$ is the $G$-subshift $\widetilde{X}$ whose configurations are those for which the restriction to every coset of $H$ is a configuration from $X$. Notice that if $X$ is a subshift defined by a forbidden set of patterns $\mathcal{F}$, then $\widetilde{X}$ is defined by exactly the same set of forbidden patterns.

All of our results are for the case where $G = H \times K$ is the direct product of two groups $H$ and $K$. In this case we identify $H$ with the subgroup $H \times \{1_K\} \leqslant G$, and thus the free extension of some $H$-subshift $X$ is the $G$-subshift $\widetilde{X}$ whose configurations satisfy that for each $k \in K$ their restriction to $H \times \{k\}$ lies in $X$.

Our first result generalizes our previous observation that no simulation result can hold for free extensions from $H =\ZZ$ to $G = \ZZ^2$ to the context where both $H$ and $G = H \times K$ are finitely generated amenable groups.

\begin{theorem}\label{thm:amenable_is_nope}
	Let $H,K$ be two infinite and finitely generated amenable groups. There exists an $H$-subshift which is effectively closed by patterns and whose free extension to $G=H \times K$ is not sofic.
\end{theorem}

A subshift on a finitely generated group is ``effectively closed by patterns'' if there exists a recursively enumerable list of pattern codings (where elements of the groups are encoded as words on the generators) which describes it. It is a natural generalization of the natural notion of effectively closed subshift for $\ZZ$, see~\cite{ABS2017}. The proof of~\Cref{thm:amenable_is_nope} is separated on two cases, one where $H$ has decidable word problem, and one where $H$ satisfies a technical property. We finally show that these two cases cover all groups. 

The main result of this article is that as soon as the first group has decidable word problem and we replace the second group by a non-amenable group, then the free extension of any effectively closed subshift is a sofic subshift. 

\begin{theorem}\label{thm:main_theorem}
	Let $H$ be a finitely generated group with decidable word problem and $N$ be a non-amenable group. The free extension of every effectively closed $H$-subshift to $G=H \times N$ is sofic.
\end{theorem}

We remark that in this result we do not require the group $N$ to be finitely generated, or even countable. For instance, we have that the free extension of every effectively closed $\ZZ$-subshift to $\ZZ \times \operatorname{SL}_2(\RR)$ is a sofic subshift.

The proof of~\Cref{thm:main_theorem} is done in two steps. First, we construct in an arbitrary non-amenable group $N$ a subshift of finite type in which every configuration encodes a collection of pairwise disjoint infinite binary trees which are rooted in every element of the group. This is done through paradoxical decompositions and is a natural generalization of the technique we introduced in~\cite{BaSaSa_2021}. This allows us to associate every coset of a group $H$ in $H \times N$ with a product with a binary tree $H \times \{0,1\}^*$. The second step of the proof constructs a rooted variant of a subshift of finite type in $H\times \{0,1\}^*$ whose projection to the empty word $H \times \{\epsilon\}$ coincides with the effectively closed subshift on $H$. This construction is the most technical aspect of this paper and is explained with details in~\Cref{sec:mainthm}.

The remainder of the paper is dedicated to exploring the applications of~\Cref{thm:main_theorem}. An immediate consequence is that if $H$ is a finitely generated group with decidable word problem and $N$ is a finitely generated non-amenable group. Then the trivial extension of every effectively closed $H$-subshift to $G=H \times N$ is also sofic (\Cref{cor:simulation_trivial_expansive}).

A second result concerns the simulation of more general effective actions than subshifts. Indeed, the initial breakthrough (which preceeded~\cite{AubrunSablik2010,DurandRomashchenkoShen2010}) in this direction was a result of Hochman~\cite{Hochman2009b}, which showed that for every effectively closed action $\ZZ\curvearrowright X$, that is, every homeomorphism of a zero-dimensional space which can be described in a precise manner through a Turing machine, there exists a $\mathbb{Z}^3$-subshift of finite type which factors onto the trivial extension of $\ZZ\curvearrowright X$ (the $\ZZ^3$-action on $X$ such that $\{0\} \times \ZZ^2$ acts trivially and $\ZZ \times \{(0,0)\}$ acts isomorphically as $\ZZ\curvearrowright X$). This result established an important bridge between computability and higher-dimensional abelian group actions, a remarkable illustration of this connection being the classification of entropies of $\ZZ^d$-subshifts of finite type by Hochman and Meyerovitch~\cite{HochmanMeyerovitch2010}.

A few results by the authors have extended the initial work of Hochman to actions by homeomorphisms of more general finitely generated groups in several ways~\cite{BS2018,Barbieri_2019_DA,BaSaSa_2021}. Here we use~\Cref{thm:main_theorem} to provide the following new simulation theorem.

\begin{theorem}\label{thm:simulation_nonamenable}
Let $H,N$ be infinite and finitely generated groups and $G = N \times H$. Let $H$ have decidable word problem and $N$ be non-amenable. For every effectively closed action $N \curvearrowright X$ there exists a $G$-subshift of finite type which factors onto the trivial extension of $N \curvearrowright X$ to $G$.
\end{theorem}

The proof of this result uses~\Cref{thm:main_theorem} in conjunction with Toeplitz codings of effectively closed sets. As a remarkable consequence of~\Cref{thm:simulation_nonamenable} we show that on every product $H \times N$ of two finitely generated groups with decidable word problem where at least one of them is non-amenable there exists a non-empty subshift of finite type for which the shift action is free. These subshifts are usually called ``strongly aperiodic''.

\begin{theorem}\label{thm:aperiodic}
	Let $N,H$ be infinite and finitely generated groups with decidable word problem and suppose that $N$ is non-amenable. The group $H \times N$ admits a nonempty strongly aperiodic subshift of finite type. 
\end{theorem}

We remark that very recently a minimal strongly aperiodic SFT in $\ZZ \times F_2$ was constructed in~\cite{AuBiHTB_2022}. While our result covers a much larger class of groups, we do not obtain minimality. 

The paper is organized as follows. In~\Cref{sec:preliminaries} we provide definitions for all the required computability and dynamical notions, as well as present the basic results that we shall use without reference later on. In~\Cref{sec:amenables} we prove~\Cref{thm:amenable_is_nope} by constructing an explicit example in each of the aforementioned cases. Next in~\Cref{sec:trees} we construct the binary tree structure in an arbitrary non-amenable group and use it to reduce the proof of~\Cref{thm:main_theorem} to the existence of an ad-hoc SFT-like structure which we call ``rooted SFT'' (\Cref{def:rootedSFT}). ~\Cref{sec:mainthm} is the core of this article, where we provide the proof of~\Cref{thm:main_theorem} by constructing explicitly the rooted SFT with the desired properties. In~\Cref{sec:aplications} we present the applications mentioned above. Finally, in~\Cref{sec:questions} we present a few questions we were not able to solve.

 \textbf{Acknowledgments}: S. Barbieri was supported by the FONDECYT grants 11200037 and 1240085. M. Sablik was supported by ANR project Difference (ANR-20-CE48-0002) and the project Computability of asymptotic properties of dynamical systems from CIMI Labex (ANR-11-LABX-0040). V. Salo was supported by the Academy of Finland project 2608073211.

\section{Preliminaries}\label{sec:preliminaries}

We denote by $\NN$ the set of non-negative integers and use the notation $A \Subset B$ to denote that $A$ is a finite subset of $B$. For a group $G$, we denote its identity by $1_{G}$, and for a word $w = w_0w_2\dots w_{k-1} \in G^*=\bigcup_{n \in \NN}G^n$, we denote by $\underline{w}$ the element of $G$ which is obtained by multiplying the $w_i$. We denote by $\epsilon$ the empty word.

The \define{word problem} of a group $G$ with respect to $S \Subset G$ is the language \[ \texttt{WP}_{S}(G) = \{ w \in S^* :  \underline{w} = 1_{G}   \}.  \]

Let $S\Subset G$ be a set which generates $G$. We say that $G$ is \define{recursively presented} if $\texttt{WP}_{S}(G)$ is a recursively enumerable language, that is, if there is a Turing machine which on input $w \in S^*$ accepts if and only if $\underline{w} = 1_{G}$. A group is said to have \define{decidable word problem} if $\texttt{WP}_{S}(G)$ is decidable. If $G$ is finitely generated the notions above do not depend upon the choice of $S$, as long as it generates $G$.

Let $G \curvearrowright X$ and $G \curvearrowright Y$ be (left) actions by homeomorphisms of a group $G$ on two compact metrizable spaces $X$ and $Y$. A map $\phi\colon X \to Y$ is called a \define{topological morphism} if it is continuous and $G$-equivariant, that is, if $\phi(gx) = g\phi(x)$ for every $g \in G$ and $x \in X$. A topological morphism which is surjective is called a \define{topological factor map}. If there exists a topological factor map $\phi\colon X \to Y$ we say that $G \curvearrowright Y$ is a \define{topological factor} of $G \curvearrowright X$, and that $G \curvearrowright X$ is a \define{topological extension} of $G \curvearrowright Y$. If the topological morphism is a bijection, we say that $G \curvearrowright X$ is \define{topologically conjugate} to $G\curvearrowright Y$. 


Given a short exact sequence $1 \to K \to G \to H \to 1$ of groups and an action $H \curvearrowright X$, the \define{trivial extension} to $G$ of $H \curvearrowright X$ is the action $G \curvearrowright X$ so that the $K$-subaction of $G \curvearrowright X$ is trivial and the quotient action $G/K \curvearrowright X$ is isomorphic to $H \curvearrowright X$. In the case where $G = H \times K$, the trivial extension to $G$ of $H \curvearrowright X$ is the action $G \curvearrowright X$ where $\{1_H\}\times K$ acts trivially and $H\times \{1_K\} \curvearrowright X$ is isomorphic to $H \curvearrowright X$ with the canonical identification between $H \times \{1_K\}$ and $H$.

\subsection{Effectively closed actions}
Let $A$ be a finite set, for instance $A = \{\symb{0},\symb{1}\}$, and endow $A^{\NN}$ with the product of the discrete topology on each coordinate (the prodiscrete topology). For a word $w = w_0w_1\dots w_{n-1} \in A^*$ we denote by $[w]$ the cylinder set of all $x \in A^{\NN}$ such that $x_i = w_i$ for $0 \leq i \leq n-1$. We say that a set $X \subset A^{\NN}$ is \define{effectively closed} if there exists a Turing machine which halts on input $w \in A^*$ if and only if $[w] \cap X = \varnothing$. Intuitively, a set $X$ is effectively closed if there is an algorithm which provides approximations of the complement of $X$ as a union of cylinders. We remark that in terms of effective descriptive set theory, effectively closed sets are called $\Pi_1^0$ sets. 


Consider a finitely generated group $G$ with a finite generating set $S$ which contains the identity. We shall provide a computability notion for actions of $G$ on effectively closed sets. Given $X \subset A^\NN$, consider the finite set $B =A^S$. Notice that elements $y\in B^{\NN}$ can be identified as maps $y\colon \NN \times S \to A$. For $s\in S$ and $y \in B^{\NN}$, we define the projection to coordinate $s$ of $y$ as the map $\pi_s y \in A^{\NN}$ given by \[  (\pi_s y) (n) = y(n)(s) \mbox{ for every } n \in \NN.\]

Consider the set $\operatorname{Rep}(G \curvearrowright X, S) \subset (A^S)^{\NN}$ given by \[ \operatorname{Rep}(G \curvearrowright X, S) = \{ y \in B^\NN : \pi_{1_{G}}y \in X, \mbox{ and for every } s \in S, \pi_s y = s\cdot  (\pi_{1_{G}}y) \}.  \]

We call $\operatorname{Rep}(G \curvearrowright X, S)$ the \define{set representation} of $G \curvearrowright X$ determined by $S$.

\begin{definition}\label{def:ECaction}
	Let $G$ be a finitely generated group and $X\subset A^{G}$. We say an action $G \curvearrowright X$ is effectively closed, if there exists a finite generating set $S\subset G$ such that the set representation $\operatorname{Rep}(G \curvearrowright X, S)$ is an effectively closed subset of $(A^S)^{\NN}$.
\end{definition}

We remark that the set representation $\operatorname{Rep}(G \curvearrowright X, S)$ can be identified with the set of maps $y \colon S \to X$ which satisfy the property that $y(s) = s\cdot y(1_G)$, and thus represents the graph of the group action on $X$ restricted to the generators. Thus in this definition we are saying that an action $G \curvearrowright X$ is effectively closed whenever for some set of generators, its graph is an effectively closed set.

To readers who are familiar with computable functions on Cantor spaces, an effectively closed action can equivalently be defined as an action of a group on an effectively closed set where for every generator $s$ the map $x \mapsto sx$ is a computable function. In particular, this implies that the notion of effectively closed action can be alternatively defined as a property that must hold for all sets of generators. For more on effectively closed actions we suggest to read Section 2.2 of~\cite{BaSaSa_2021}.

\subsection{Shift spaces}

Let $\ag$ be a finite set and $G$ be a group. The set $\ag^G = \{ x\colon G \to \ag\}$ equipped with the left \define{shift} action $G \curvearrowright \ag^{G}$ by left multiplication given by 
\[ gx(h) \isdef x(g^{-1}h) \qquad \mbox{  for every } g,h \in G 
\mbox{ and } x \in \ag^G, \]
 is the \define{full $G$-shift}. The elements $a \in \ag$ and $x \in \ag^G$ are called \define{symbols} and \define{configurations} respectively. Given $F\Subset G$, a \define{pattern} with support $F$ is an element $p \in \ag^F$. We denote the cylinder generated by $p$ by $[p] = \{ x \in A^{G} : x|_F = p \}$ and note that the cylinders are a clopen base for the prodiscrete topology on $A^{G}$. 
 
\begin{definition}
	A subset $X \subset \ag^G$ is a \define{$G$-subshift} if and only if it is $G$-invariant and closed in the prodiscrete topology. 
\end{definition}

If the context is clear, we drop the $G$ from the notation and speak plainly of a subshift.  Equivalently, $X$ is a subshift if and only if there exists a set of forbidden patterns $\FF$ such that \[X=X_\FF \isdef \{ x \in A^{G} : gx\notin [p] \mbox{ for every } g \in G, p \in \FF  \}.\]

\begin{definition}
    Let $G$ be a group and $H \leqslant G$ a subgroup. For an $H$-subshift $X\subset A^{H}$ we define its \define{free extension} to $G$ as the subshift $\widetilde{X}\subset A^{G}$ given by \[ \widetilde{X} = \{ \widetilde{x}\in A^{G} : \mbox{ for every } g \in G, \{x(gh)\}_{h \in H} \in X \}. \]
\end{definition}

We remark that if $X\subset A^{H}$ is given by a set of forbidden patterns $\mathcal{F}$, then its free extension $\widetilde{X}$ is the $G$-subshift given by the same set of forbidden patterns. Notice that every configuration on $\widetilde{X}$ can be understood as bundles of configurations of $X$ indexed by the cosets of $H$.

An action $G \curvearrowright X$ on a compact metrizable space is expansive if there is a constant $C>0$ so that whenever $d(gx,gy)< C$ for every $g \in G$ then $x=y$. It is a well known elementary fact that an action $G \curvearrowright X$ on a closed subset $X \subset \{\symb{0},\symb{1}\}^{\NN}$ is topologically conjugate to a $G$-subshift if and only if the action is expansive.

\begin{definition}
	A subshift $X \subset \ag^G$ is a \define{subshift of finite type (SFT)} if there exists a finite set of forbidden patterns $\FF$ such that $X = X_{\FF}$.
\end{definition}

\begin{definition}
	A subshift $X \subset \ag^G$ is a \define{sofic subshift} if it is the topological factor of an SFT.
\end{definition}

It will be useful to describe topological morphisms between subshifts in an explicit way. The following classical theorem provides said description. A modern proof for actions of countable groups may be found in~\cite[Theorem 1.8.1]{ceccherini-SilbersteinC09}.

\begin{theorem}[Curtis-Lyndon-Hedlund~\cite{Hedlund1969}]\label{thm:curtis_lyndon_hedlund}
	Let $G$ be a countable group and $X \subset A^G, Y \subset B^G$ be two $G$-subshifts. A map $\phi\colon X \to Y$ is a topological morphism if and only if there exists a finite set $F \Subset G$ and $\Phi\colon A^F \to B$ such that for every $x \in X$ and $g \in G$ then $(\phi(x))(g) = \Phi((g^{-1}x)|_{F})$.
\end{theorem}

In the above definition, a map $\phi\colon X \to Y$ for which there is $\Phi\colon A \to B$  so that $(\phi(x))(g) = \Phi(x(g))$ for every $g \in G$ is called a \define{$1$-block map}. If $Y$ is the topological factor of an SFT $X$, it is always possible to construct an SFT $Z$ which is topologically conjugate to $X$ and a $1$-block topological factor map $\widetilde{\phi}\colon Z \to Y$. Furthermore, if $G$ is generated by a finite set $S$, one can ask that $Z$ is \define{nearest neighbor} with respect to $S$, that is, it is described by a set of forbidden patterns whose supports are of the form $\{1_{G},s\}$ where $s$ is a generator of $G$. A proof of this elementary fact can be found on~\cite[Proposition 1.7]{tesis}.

\subsection{Effectively closed subshifts}

It is often useful to give an explicit notion of effectively closed in order to be able to work with the space $A^{G}$ instead of $A^{\NN}$. Let $G$ be a finitely generated group, $S$ a finite set of generators for $G$ and $\ag$ an alphabet. Recall that for $w \in S^*$, we denote by $\underline{w}$ its corresponding element on $G$. For any $W \Subset S^*$, a map $c\colon W \to \ag$ is called a \define{pattern coding}. The \define{cylinder} defined by a pattern coding $c$ is given by
	\[ [c] = \bigcap_{w \in W} \{x \in A^G: x(\underline{w}) = c(w)  \} .  \]

A pattern coding can be thought of as a pattern on the free monoid $S^*$. It can be represented on the tape of a Turing machine as a finite sequence of tuples $\{(w_1,a_1),(w_2,a_2),\dots,(w_k,a_k)\}$ where $w_i \in S^*, a_i \in \ag$ and $c(w_i)=a_i$. A set $\mathcal{C}$ of pattern codings defines a $G$-subshift $X_{\mathcal{C}}$ by setting \[X_{\mathcal{C}} \isdef \{ x \in A^{G} : gx\notin [c] \mbox{ for every } g \in G, c \in \mathcal{C} \}.\]

\begin{definition}
	A $G$-subshift $X$ is \define{effectively closed by patterns} if there exists a recursively enumerable set of pattern codings $\mathcal{C}$ such that $X = X_{\mathcal{C}}$. 
\end{definition}

A Turing machine which accepts a pattern coding $c$ if and only if $c \in \mathcal{C}$ is said to \define{define} $X = X_{\mathcal{C}}$. Note that neither $\mathcal{C}$ nor the Turing machine which defines $X$ are unique.

It is not true in general that subshifts which are effectively closed by patterns coincide up to topological conjugacy with effectively closed expansive actions. Every effectively closed expansive action is topologically conjugate to a subshift which is effectively closed by patterns, but the converse only holds if the group is recursively presented, see~\cite[Proposition~2.16]{BaSaSa_2021}. Therefore if we consider recursively presented groups (or groups with decidable word problem) we will simply refer to an ``effectively closed subshift'', while beyond this class we will add ``by patterns'' if we mean this notion instead of ``topologically conjugate to an expansive effectively closed action''.



\section{Product of two amenable groups}\label{sec:amenables}

In this section we prove~\Cref{thm:amenable_is_nope}. We will proceed through the construction of two explicit instances. One of them will work whenever the group $H$ has decidable word problem, whereas the second one will work when $H$ satisfies a technical condition which we call ``property (S)''. Finally, we will show that these two cases cover all finitely generated groups.

Both of the constructions we present here are variations of a well-known example in the literature due to Jeandel (unpublished) and called the ``mirror shift''. This is an effectively closed $\ZZ^2$-subshift which is not sofic and which has the property that one of its (non-expansive) $\ZZ$-subactions can not be realized as the topological factor of a subaction of any $\ZZ^2$-SFT. An interested reader can find more about this mirror shift by reading~\cite[Section 2.4]{ABS2017} and M. Hochman's chapter in~\cite{CANT2010}.

For the remainder of this section we fix two infinite and finitely generated groups $H$ and $K$, and finite symmetric sets of generators $S_H,S_K$ which contain the identity for $H$ and $K$ respectively. Consider the word metric in $H$ with respect to $S_H$. For $n \in \NN$, denote by $B_n = \{ g \in H : |g| \leq n \}$ the set of all elements at distance at most $n$ from the identity in $H$. 

For finite sets $T \Subset H$ and $U \Subset K$, define $\partial T = TS_H \setminus T$, $\partial U = US_K \setminus U$ and $\partial (T\times U) = (TS_H \times US_K) \setminus (T \times U)$. Remark that as $S_H,S_K$ were chosen so that they contain the identity, we have that $|\partial T| = |TS_H|-|T|$, $|\partial U| = |US_K|-|U|$ and $|\partial(T \times U)| = |\partial T||U|+|\partial U||T| + |\partial T||\partial U|$.

\subsection{Property (S): the reflection shift}

Let $G$ be a group which is finitely generated by a symmetric set $S \Subset G$ and let $A = \{ \ast,\symb{0},\symb{1}\}$. The \define{reflection shift} $X_{\texttt{R}} \subset A^{G}$ is the set of configurations determined by the set of forbidden pattern codings \[ \mathcal{C} = \{ c : \{\epsilon, w, w^{-1}\} \to A \mbox{ such that } w \in S^*, c(\epsilon)= \ast, \mbox{ and } c(w) \neq c(w^{-1})\}.  \] 

In simpler words, the reflection shift consists of all configurations $x \in A^{G}$ such that if $x(g) =\ast$ for some $g \in G$, then $x(gh) = x(gh^{-1})$ for every $h \in G$, see~\Cref{fig:inversemirror}. Let us notice that this subshift is always nonempty, as $\{\symb{0},\symb{1}\}^{G} \subset X_{\texttt{R}}$. It is also clear that the set of forbidden pattern codings $\mathcal{C}$ is recursively enumerable, and thus $X_{\texttt{R}}$ is effectively closed by patterns.

\begin{figure}[ht!]
	\scalebox{1}{
	\xConfig{
		\symb{0} \& \symb{1} \& \symb{1} \& \symb{1} \& \symb{0} \& \symb{1} \& \symb{0} \& \symb{0} \& \symb{1} \& \symb{1} \& \symb{1} \& \symb{0} \& \symb{1} \\
		\symb{1} \& \symb{1} \& \symb{0} \& \symb{1} \& \symb{0} \& \symb{1} \& \symb{1} \& \symb{0} \& \symb{1} \& \symb{1} \& \symb{0} \& \symb{0} \& \symb{1} \\
		\symb{0} \& \symb{1} \& \symb{1} \& \symb{0} \& \symb{0} \& \symb{0} \& \symb{0} \& \symb{0} \& \symb{1} \& \symb{0} \& \symb{0} \& \symb{0} \& \symb{0} \\
		\symb{1} \& \symb{1} \& \symb{0} \& \symb{0} \& \symb{0} \& \symb{1} \& \symb{0} \& \symb{0} \& \symb{1} \& \symb{0} \& \symb{0} \& \symb{0} \& \symb{1} \\
		\symb{1} \& \symb{1} \& \symb{0} \& \symb{0} \& \symb{0} \& \symb{1} \& \symb{0} \& \symb{0} \& \symb{1} \& \symb{0} \& \symb{0} \& \symb{0} \& \symb{1} \\
		\symb{0} \& \symb{1} \& \symb{0} \& \symb{0} \& \symb{0} \& \symb{1} \& \symb{0} \& \symb{0} \& \symb{1} \& \symb{0} \& \symb{1} \& \symb{1} \& \symb{1} \\
		\symb{1} \& \symb{0} \& \symb{0} \& \symb{0} \& \symb{1} \& \symb{0} \& \symb{\ast} \& \symb{0} \& \symb{1} \& \symb{0} \& \symb{0} \& \symb{0} \& \symb{1}\\
		\symb{1} \& \symb{1} \& \symb{1} \& \symb{0} \& \symb{1} \& \symb{0} \& \symb{0} \& \symb{1} \& \symb{0} \& \symb{0} \& \symb{0} \& \symb{1} \& \symb{0} \\
		\symb{1} \& \symb{0} \& \symb{0} \& \symb{0} \& \symb{1} \& \symb{0} \& \symb{0} \& \symb{1} \& \symb{0} \& \symb{0} \& \symb{0} \& \symb{1} \& \symb{1} \\
		\symb{1} \& \symb{0} \& \symb{0} \& \symb{0} \& \symb{1} \& \symb{0} \& \symb{0} \& \symb{1} \& \symb{0} \& \symb{0} \& \symb{0} \& \symb{1} \& \symb{1} \\
		\symb{0} \& \symb{0} \& \symb{0} \& \symb{0} \& \symb{1} \& \symb{0} \& \symb{0} \& \symb{0} \& \symb{0} \& \symb{0} \& \symb{1} \& \symb{1} \& \symb{0} \\
		\symb{1} \& \symb{0} \& \symb{0} \& \symb{1} \& \symb{1} \& \symb{0} \& \symb{1} \& \symb{1} \& \symb{0} \& \symb{1} \& \symb{0} \& \symb{1} \& \symb{1} \\
		\symb{1} \& \symb{0} \& \symb{1} \& \symb{1} \& \symb{1} \& \symb{0} \& \symb{0} \& \symb{1} \& \symb{0} \& \symb{1} \& \symb{1} \& \symb{1} \& \symb{0} \\
	}\;}
\caption{Part of a configuration of the reflection shift in $\ZZ^2$.}
\label{fig:inversemirror}
\end{figure}

The following notion was introduced and discussed in a mathoverflow thread by one of the authors\footnote{\href{https://mathoverflow.net/questions/370239/splendid-groups}{https://mathoverflow.net/questions/370239/splendid-groups}} in the context of locally compact groups. In said reference the groups with the property are called ``splendid''. Here we give them the name ``property (S)''.

\begin{definition}
	A countable group $G$ satisfies property (S) if for every finite subset $F \Subset G$ there exists $g \in G$ such that \[ gFg \cap F = \varnothing.  \]
\end{definition}

\begin{proposition}\label{prop:case_prop_S}
	Let $H,K$ be infinite, finitely generated  and amenable groups and suppose that $H$ satisfies property (S). Then the free extension of the reflection shift on $H$ to $G = H \times K$ is not sofic.
\end{proposition}

\begin{proof}
	 Let us denote by $\widetilde{X}_{\texttt{R}} \subset \{ \symb{0},\symb{1},\ast \}^{H \times K}$ the free extension to $H \times K$ of the reflection shift on $H$. Let $Y \subset \widetilde{X}_{\texttt{R}}$ be the subshift given by \begin{align*}
	     Y & = \{ y \in \widetilde{X}_{\texttt{R}} : \mbox{ for every } h \in H, k \in K, s \in S_K \mbox{ we have } y(h,k)= \ast \mbox{ if and only if } y(h,ks) = \ast    \}\\
      & = \{ y \in \widetilde{X}_{\texttt{R}} : \mbox{ for every } h \in H, k_1,k_2 \in K, \mbox{ we have } y(h,k_1)= \ast \mbox{ if and only if } y(h,k_2) = \ast    \}.
	 \end{align*} 

  The equality between both definitions of $Y$ is due to the fact that $S_K$ is a generating set for $K$. We can think of $Y$ as the subset of all configurations of $\widetilde{X}_{\texttt{R}}$ in which the symbols $\ast$ occur in ``columns'' in the second group. That is, if the symbol $\ast$ occurs in some position $(h,k)$, then it occurs in all $\{h\}\times K$.
	 
	 Suppose that $\widetilde{X}_{\texttt{R}}$ is sofic. It follows directly from the definition that $Y$ is also sofic. Therefore there exists a finite set $B$, a nearest neighbor SFT $Z\subset B^{H \times K}$ and a 1-block topological factor map $\phi \colon Z \to Y$.
	 
	Now let $\varepsilon >0$ such that \[ 2\varepsilon + \varepsilon^2 < \frac{\log(2)}{\log(|B|)}. \]
	As both $H,K$ are amenable, there exists finite sets $T\Subset H$ and $U\Subset K$ such that $|\partial T|\leq \varepsilon|T|$ and $|\partial U|\leq \varepsilon|U|$. We will further choose $T$ such that $T = T^{-1}$. This can always be done, see~\cite[Corollary 5.3]{Namioka1964} From our choice of $\varepsilon$ we obtain that \[2^{|T\times U|} = 2^{|T||U|} > |B|^{|U||T|(2\varepsilon + \varepsilon^2)} \geq |B|^{|T||\partial U|+|U||\partial T|+|\partial T||\partial U|} =  |B|^{|\partial (T \times U)|}.  \]
	
	By our assumption that $H$ has property (S), we obtain that there is $g \in H$ such that $gTg \cap T = \varnothing$ and thus as $T$ was chosen symmetric we obtain that the sets $gT$ and $T^{-1}g^{-1}$ are disjoint. Notice that this implies that neither $gT$ nor $T^{-1}g^{-1}$ may contain the identity.
	
	For any pattern $p \colon gT \times U \to \{\symb{0},\symb{1}\}$ we construct a configuration $y_p \in Y$ whose restriction to $gT\times U$ is $p$ by letting\[ y_p(h,k) = \begin{cases}
		\ast & \mbox{ if } h = 1_{H}\\
		p(h,k) & \mbox{ if } h \in gT, k \in U\\
		p(h^{-1},k) & \mbox{ if } h \in T^{-1}g^{-1}, k \in U\\
		\symb{0} & \mbox{otherwise}.\\
	\end{cases}  \]

	For each $p$ as above, let $x_p \in Z$ such that $\phi(x_p) = y_p$. Notice that there are $2^{|T||U|}$ patterns $p$ as above but there are at most $|B|^{|\partial (T \times U)|}$ possible restrictions of a configuration in $Z$ to $\partial(gT \times U)$ and thus by the pigeonhole principle there must exist two distinct patterns $p,p'$ such that \[ x_p|_{\partial(gT \times U)} = x_{p'}|_{\partial(gT \times U)}.  \]
	Fix some pair $(t,u)\in gT \times U$ such that $p(t,u)\neq p'(t,u)$. As $Z$ is nearest neighbor, it follows that we may paste the restriction of $x_{p'}$ at positions $(H \setminus gT) \times (K \setminus U)$ and $x_{p}$ at $gT \times U$ and obtain a valid configuration in $Z$. Applying the map $\phi$ to said configuration we obtain $y \in \{\symb{0},\symb{1},\ast\}^{H \times K}$ which satisfies that $y(1_{H},u)=\ast$, $y(t,u) = y_{p}(t,u) = p(t,u)$ but $y(t^{-1},u) = y_{p'}(t^{-1},u) = y_{p'}(t,u) =p'(t,u)$. As $p(t,u)\neq p'(t,u)$ we obtain that $y \notin Y$, which contradicts the assumption that $\phi$ is a topological factor map. We conclude that $\widetilde{X}_{\texttt{R}}$ cannot be a sofic subshift.\end{proof}
	
	\subsection{Decidable word problem: the ball mimic shift} Now we shall deal with the case where $H$ has decidable word problem. 
	
	\begin{lemma}\label{lemma:twosequences}[Lemma 2.15 of~\cite{ABS2017}]
		Suppose the word problem of $H$ is decidable. There exists two recursively enumerable sequences $\boldsymbol{u}=(u_n)_{n \in \NN}$  and $\boldsymbol{v}=(v_n)_{n \in \NN}$ of words in $S_H^*$ such that the collection of sets \[ \{\{1_H\}\} \cup \{ \underline{u_n}B_n : n \in \NN\} \cup \{ \underline{v_n}B_n : n \in \NN \}  \mbox{ is pairwise disjoint}.          \] 
	\end{lemma}
	
	In simpler words, there is an algorithm which enumerates two sequences of words which represent ``centers'' of pairwise disjoint balls of increasing radius such that no ball contains the identity. The proof of~\Cref{lemma:twosequences} is straightforward and an explicit algorithm producing such sequences (using as a subalgorithm one for the word problem of $H$) can be found in~\cite{ABS2017}. 
	
	Let $A = \{ \ast,\symb{0},\symb{1}\}$ and fix two sequences $\boldsymbol{u},\boldsymbol{v}$ as in~\Cref{lemma:twosequences}. For each $n \in \NN$ let $a_n = \underline{u_n}$ and $b_n = \underline{v_n}$ be the corresponding elements of $H$. The \define{ball mimic shift} $X_{\texttt{BM}} \subset A^{H}$ is the set of configurations $x \in A^{H}$ which satisfy the following property: if $x(g) = \ast$ for some $g \in H$, then for every $n \in \NN$ and $h \in B_n$ we have $x(ga_nh) = x(gb_nh)$.
	
	In simpler words, if the symbol $\ast$ occurs at some position $g$ in $x$, then the restriction of $x$ to $ga_nB_n$ must mimic the restriction of $x$ to $gb_nB_n$ for every $n \in \NN$. 
	
	\begin{lemma}
		Let $H$ be finitely generated group with decidable word problem. The ball mimic shift with respect to two sequences as in~\Cref{lemma:twosequences} is effectively closed (by patterns).
	\end{lemma}
	
	\begin{proof}
		Let $\boldsymbol{u}=(u_n)_{n \in \NN}$  and $\boldsymbol{v}=(v_n)_{n \in \NN}$ the sequences that define $X_{\texttt{BM}}$. Consider the set of pattern codings\[ \mathcal{C} = \{ c : \{\epsilon, u_nw, v_nw\} \to A \mbox{ such that } n \in \NN, w \in S^*, |w| \leq n, c(\epsilon)= \ast, \mbox{ and } c(u_nw) \neq c(v_nw)\}.  \] 
		As $\boldsymbol{u},\boldsymbol{v}$ are recursively enumerable, it follows that $\mathcal{C}$ is a recursively enumerable set of pattern codings. It is immediate that $X_{\texttt{BM}} = X_{\mathcal{C}}$.
	\end{proof}
	
	The next proof is very similar to the proof of~\Cref{prop:case_prop_S}, we give it for completeness.
	
	\begin{proposition}\label{prop:case_WP_dec}
		Let $H,K$ be infinite, finitely generated and amenable groups and suppose that $H$ has decidable word problem. The free extension of the ball mimic shift on $H$ to $G = H \times K$ is not sofic.
	\end{proposition}
	
	\begin{proof}
		Let us denote by $\widetilde{X}_{\texttt{BM}} \subset \{ \symb{0},\symb{1},\ast \}^{H \times K}$ the free extension to $H \times K$ of the ball mimic shift on $H$. Let $Y \subset \widetilde{X}_{\texttt{BM}}$ be the subshift given by \begin{align*}
		    Y = \{ y \in \widetilde{X}_{\texttt{BM}} : \mbox{ for every } h \in H, k_1,k_2 \in K \mbox{ we have } y(h,k_1)= \ast \mbox{ if and only if } y(h,k_2) = \ast    \}.
		\end{align*}    
		
		Suppose that $\widetilde{X}_{\texttt{BM}}$ is sofic. As $K$ is finitely generated, it follows that $Y$ is sofic and thus there exists a finite set $B$, a nearest neighbor SFT $Z\subset B^{H \times K}$ and a 1-block topological factor map $\phi \colon Z \to Y$.
		
		Let $\varepsilon >0$ such that \[ 2\varepsilon + \varepsilon^2 < \frac{\log(2)}{\log(|B|)}. \]
		As both $H,K$ are amenable, there exists finite sets $T\Subset H$ and $U\Subset K$ such that $|\partial T|\leq \varepsilon|T|$, $|\partial U|\leq \varepsilon|U|$ and $T = T^{-1}$. From our choice of $\varepsilon$ we obtain that \[ 2^{|T||U|} > |B|^{|U||T|(2\varepsilon + \varepsilon^2)} \geq |B|^{|T||\partial U|+|U||\partial T|+|\partial T||\partial U|} =  |B|^{|\partial (T \times U)|}.  \]

		As $T$ is finite, we can find $n \in \NN$ such that $T \subset B_n$. Let $p \colon T \times U \to \{\symb{0},\symb{1}\}$ be any pattern. As $a_nB_n \cap b_nB_n = \varnothing$ and $1_H \notin a_nB_n \cup b_nB_n$, we can construct the following map $y_p \colon H \times K \to A$.\[y_p(h,k) = \begin{cases}
			\ast & \mbox{ if } h = 1_H\\
			p(t,k) &\mbox{if } k \in U  \mbox{ and } h \in \{a_nt,b_nt\} \mbox{ for some } t \in T \\
			\symb{0} &\mbox{otherwise}
		\end{cases}\]
		
		It is clear by construction that $y_p \in Y$. For each $p$ as above, let $x_p \in Z$ such that $\phi(x_p) = y_p$. Notice that there are $2^{|T||U|}$ patterns as above but there are at most $|B|^{|\partial (T \times U)|}$ possible restrictions of a configuration in $Z$ to $\partial(a_nT \times U)$ and thus by the pigeonhole principle there must exist two distinct patterns $p,p'$ such that \[ x_p|_{\partial(a_nT \times U)} = x_{p'}|_{\partial(a_nT \times U)}.  \]
		Fix some pair $(t,u)\in T \times U$ such that $p(t,u)\neq p'(t,u)$. As $Z$ is nearest neighbor, it follows that we may paste the restriction of $x_{p'}$ at positions $(H \setminus a_nT) \times (K \setminus U)$ and $x_{p}$ at $a_nT \times U$ and obtain a valid configuration in $Z$. Applying the map $\phi$ to said configuration we obtain $y \in \{\symb{0},\symb{1},\ast\}^{H \times K}$ which satisfies that $y(1_{H},u)=\ast$, $y(a_nt,u) = y_{p}(a_nt,u) = p(t,u)$ but $y(b_nt,u) = y_{p'}(b_nt,u) =p'(t,u)$. As $p(t,u)\neq p'(t,u)$ we obtain that $y \notin Y$, which contradicts the assumption that $\phi$ is a topological factor map. We conclude that $\widetilde{X}_{\texttt{BM}}$ cannot be sofic.\end{proof}
	
	\subsection{Proof of~\Cref{thm:amenable_is_nope}}

	By~\Cref{prop:case_WP_dec,prop:case_prop_S}, we know that if $H$ is a finitely generated amenable group which has decidable word problem or satisfies property (S), then there exists an $H$-subshift which is effectively closed by patterns and whose free extension to $H\times K$ is not sofic.
	
	We will now show that any finitely generated group which does not satisfy property (S) must be virtually nilpotent. As every finitely generated virtually nilpotent group has decidable word problem (see for instance~\cite[Theorem 4.6]{LyndonSchupp1977}), this is enough to prove~\Cref{thm:amenable_is_nope}.
	
	Let us recall that a group $G$ has the infinite conjugacy class property (ICC) if every non-trivial element has infinitely many elements in its conjugacy class. A classical result of Duguid and McLane~\cite{mclain_1956,duguid_mclain_1956} (see~\cite{FrischFerdowsi_ICC_quot} for a modern proof by Frisch and Vahidi Ferdowsi) states that every infinite finitely generated group is either virtually nilpotent or it admits an ICC quotient.
	
	\begin{lemma}\label{lem:quot_prop_S}
		If a group $G$ admits a quotient $H$ with property (S), then $G$ has property (S).
	\end{lemma}
	
	\begin{proof}
		Let $\phi \colon G \to H$ be an epimorphism and let $F\Subset G$. It follows that $\phi(F)$ is also finite and thus there exists $h \in H$ such that $h\phi(F)h \cap \phi(F)=\varnothing$. Choose $g\in G$ such that $\phi(g)=h$, it follows that $\phi(gFg) = h\phi(F)h$ and thus $gFg \cap F = \varnothing$ as well.
	\end{proof}
	
	Therefore, it suffices to show that groups with the ICC property have property (S). This is itself a consequence of a result of Erschler and Kaimanovitch~\cite{Erschler2023}.
	
	\begin{proposition}\label{prop_erschler}[Proposition 4.25 of~\cite{Erschler2023}]
		Let $G$ have the ICC property. For every finite $Z\Subset G$ there exist infinitely many $g \in G$ such that the only solutions to the equation \[ gxg^{\varepsilon} = y \]
		with $\varepsilon\in \{-1,+1\}$ and $x,y \in Z$, are given (if they exist) by $x = y = 1_{G}$.
	\end{proposition}
	
	With this proposition in hand, we can show that ICC groups have property (S).
	
	\begin{lemma}\label{lema:ICC_S}
		Let $G$ be a finitely generated ICC group. Then $G$ has property (S).
	\end{lemma}
	
	\begin{proof}
		Let $F\Subset G$. As $G$ is finitely generated, there is $n \in \NN$ such that $F \subset B_n$, where $B_n$ denotes the ball of size $n$ in the word metric with respect to some set of generators. Set $Z = B_{2n+1}$. By~\Cref{prop_erschler}, there exists $g \in G$ such that the only possible solution to $gxg=y$ are given by $x=y=1_{G}$. Note that if there are any such solutions, then $g$ is necessary an involution.
		
		If $g$ is not an involution, it follows that $gFg \cap F = \varnothing$ and we are done. If $g$ is an involution choose any $h$ such that $|h|=n+1$ and let $g' = gh$. Suppose there are $k,k' \in B_n$ such that $g'kg' = k'$, then $g(hk)g = k'h^{-1}$. As $\max\{|hk|,|k'h^{-1}|\} \leq 2n+1$, it follows that both $hk,k'h^{-1}\in Z$ and thus $hk = k'h^{-1} = 1_{G}$, which cannot happen because $|h| = n+1$ and $|k| \leq n$. We conclude that $g'B_ng' \cap B_n = \varnothing$ and thus $g'Fg' \cap F = \varnothing$.
	\end{proof}
	
	\begin{proposition}\label{prop:prop_S_nope_is_vn}
		Let $G$ be a group without property (S). Then $G$ is virtually nilpotent.
	\end{proposition}
	
	\begin{proof}
		If $G$ is not virtually nilpotent, then it admits a quotient with the ICC property. By~\Cref{lema:ICC_S} it follows that $G$ admits a quotient with property (S). By~\Cref{lem:quot_prop_S}, we conclude that $G$ has property (S).
	\end{proof}
	
	As every finitely generated virtually nilpotent group has decidable word problem,~\Cref{thm:amenable_is_nope} follows from~\Cref{prop:prop_S_nope_is_vn} and the considerations above.

\section{Growing trees in non-amenable groups}\label{sec:trees}

Let $\kappa \geq 2$ be an integer. It is well known that a group $N$ is non-amenable if and only if one can find a finite subset $K\Subset N$ and a $\kappa$-to-$1$ surjective map $\varphi \colon N \to N$ such that $g^{-1}\varphi(g) \in K$ for every $g \in N$, see for instance~\cite[Theorem 4.9.2]{ceccherini-SilbersteinC09}. The main observation from this section, and which was the main tool in~\cite{BaSaSa_2021}, is that the space of all such maps, which can be thought of as a parametrization of a set of paradoxical decompositions of $N$, can be encoded as a subshift of finite type.

The purpose of this section is to exploit this observation to construct a subshift of finite type on which every configuration encodes a map defined by local rules which associates to every element of the group a binary tree, and the binary trees are pairwise disjoint. This will enable us to treat a non-amenable group $N$ as if each element were an infinite binary tree.
More precisely, endow both $N$ and $\{\symb{0},\symb{1}\}^{*}$ with the discrete topology, the following result is the main technical tool of our reduction.

\begin{lemma}\label{lem:arbolitos}
	Let $N$ be a non-amenable group. There exists a nonempty $N$-subshift of finite type $\mathbf{T}$ and continuous maps $\texttt{root}, \texttt{son}_0, \texttt{son}_1$ from $\mathbf{T}$ to a finite subset of $N$. These maps induce a continuous map $\gamma \colon \mathbf{T} \times \{\symb{0},\symb{1}\}^{*}\times N \to N$ which is defined recursively by:
 
    \begin{itemize}
		\item $\gamma(\tau,\epsilon,g) = g \cdot \texttt{root}(g^{-1}\tau)$ for every $g\in N$, $\tau \in \mathbf{T}$.
		\item $\gamma(\tau,wb,g) = \gamma(\tau,w,g) \cdot \texttt{son}_b( (\gamma(\tau,w,g))^{-1}\tau)$, for every $g \in N$,  $\tau \in \mathbf{T}$, $w \in \{0,1\}^*$ and $b \in \{0,1\}$.
	\end{itemize}

The map $\gamma\colon \mathbf{T}\times \{\symb{0},\symb{1}  \}^* \times N \to N$ is called the \define{tree map} and satisfies the following two properties:
 
	\begin{enumerate}
		\item for every $\tau \in \mathbf{T}$, $w\in \{\symb{0},\symb{1}\}^{*}$ and $g \in N$, we have $\gamma(\tau,w,g) = g\gamma(g^{-1}\tau,w,1_{N})$.
		\item for every $\tau \in \mathbf{T}$ the map $\gamma_{\tau}\colon \{0,1\}^{*}\times N \to N$ is injective, where $\gamma_{\tau}(w,g) = \gamma(\tau,w,g)$ for every $w\in \{0,1\}^*$ and $g \in N$.
	\end{enumerate}
\end{lemma}

The intuition is that the map \texttt{root} gives the root element of the tree associated to a given node, while $\texttt{son}_0, \texttt{son}_1$ give the directions to go down the tree. 

\subsection{The binary tree shift} For the remainder of the section, fix a non-amenable group $N$ and a finite symmetric set $K\Subset N$ such that there exists a $3$-to-$1$ map $\varphi \colon N \to N$ such that $g^{-1}\varphi(g)\in K$ for every $g \in N$. We also fix the finite set $A = K^4\times \{\symb{0},\symb{1},\symb{2}\}$. For $a = (k_0,k_1,k_2,k_3,i)\in A$, we write $s_0(a)=k_0, s_1(a)=k_1, s_2(a)=k_2, p(a)=k_3$ and  $c(a)=i$.

\begin{definition}
	 The binary tree shift $\mathbf{T}\subset A^{N}$ is the set of all configurations $\tau\in A^N$ which satisfy the following two constraints for every $g \in N$:
	\begin{enumerate}
		\item if we let $k = p(\tau(g))$ and $i = c(\tau(g))$, then $s_i(\tau(gk)) = k^{-1}$.
		\item For every $i \in \{\symb{0},\symb{1},\symb{2}\}$ if we have $s_i(\tau(g))=k$ then $p(\tau(gk))=k^{-1}$ and $c(\tau(gk)) = i$.
	\end{enumerate}
\end{definition}

The elements $s_i(a)$ are locally encoding the preimages of a $3$-to-$1$ map, while $p(a)$ locally encodes the map itself. The colour $c(a)$ is essentially partitioning the group in three disjoint sets such that the restriction maps are bijections. The two rules above are simply encoding the fact that this local information is consistent.

It is clear that $\mathbf{T}$ is an $N$-subshift of finite type. We argue that $\mathbf{T}$ it is nonempty. Let $\varphi \colon N \to N$ be a $3$-to-$1$ map such that $g^{-1}\varphi(g)\in K$ for every $g \in N$. Partition $N$ as a disjoint union $N_0 \cup N_1 \cup N_2$ such that for every $i \in \{0,1,2\}$ the map $\varphi_i \colon N_i \to N$ given by $\varphi_i = \varphi|_{N_i}$ is a bijection. For $g \in N$ let $i_g \in \{0,1,2\}$ such that $g \in N_{i_g}$. Define $\tau \in A^{N}$ through \[ \tau(g) = (g^{-1}\varphi_0^{-1}(g),g^{-1}\varphi_1^{-1}(g),g^{-1}\varphi_2^{-1}(g),g^{-1}\varphi(g),i_g) \mbox{ for every } g \in N.  \]
From our choice that $K$ is symmetric it follows that $\tau(g) \in A$. It is clear from this construction that $\tau$ satisfies both constraints from the definition and thus $\mathbf{T}$ it is nonempty.

Now we define the three local maps from~\Cref{lem:arbolitos}. Let $\tau \in \mathbf{T}$ and let $a = \tau(1_{N})$ be its value at the identity of $N$. Recall that we write $a = (s_0(a),s_1(a),s_2(a),p(a),c(a))$.
\begin{enumerate}
	\item The map $\texttt{root}\colon \mathbf{T} \to K$ is given by \[ \texttt{root}(\tau) = s_{c(a)+2 \bmod{3}}(a).  \]
	\item The maps $\texttt{son}_b\colon \mathbf{T} \to K$ for $b \in \{0,1\}$ are given by \[ \texttt{son}_i(\tau) = s_{c(a)+b \bmod{3}}(a).  \]
\end{enumerate}

Now that these local maps are defined, let $\gamma\colon \mathbf{T}\times \{\symb{0},\symb{1}  \}^* \times N \to N$ be defined as in the statement of~\Cref{lem:arbolitos}. The construction of $\gamma$ is illustrated in~\Cref{fig:arbolitos}. Continuity of $\gamma$ is immediate from the locality of the maps $\texttt{root}$, $\texttt{son}_0$ and $\texttt{son}_1$. Furthermore, we remark that we may traverse the tree using local information. Namely, if $w = ub$ for some $u \in \{\symb{0},\symb{1}  \}^*$, we have the relations $\gamma(\tau,u,g) = \gamma(\tau,w,g)p(\tau(\gamma(\tau,w,g)))$ and $\gamma(\tau,\epsilon,g)p(\gamma(\tau,\epsilon,g)) = g$.

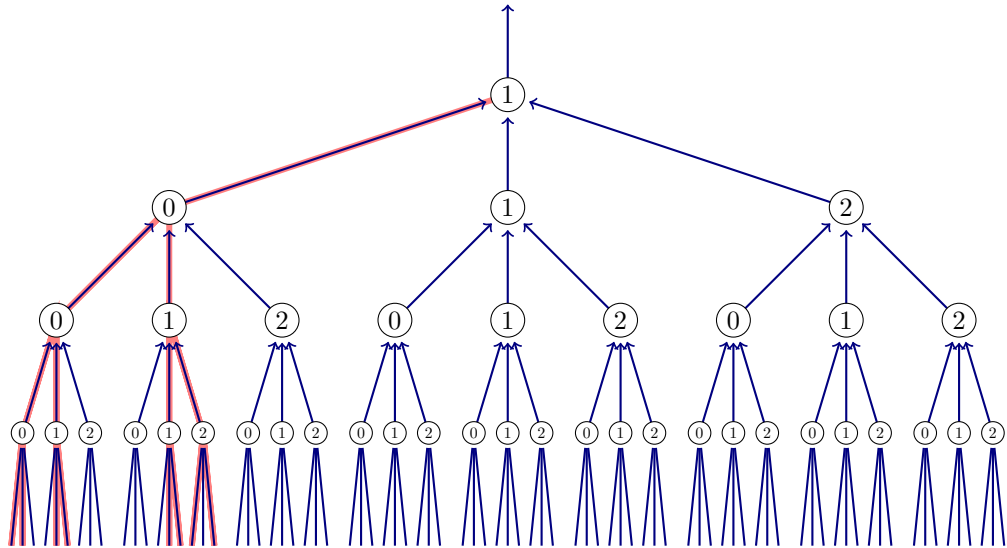
\begin{figure}[ht!]
	\centering\begin{tikzpicture}[scale=1]
		\begin{scope}[scale = 0.75, shift={(0,0)},rotate=0]
			\draw[line width=0.8mm,red, opacity = 0.5, cap=round] (-6,-2) -- (-8,-4) -- (-8.6,-6) -- (-8.8,-8) -- (-8.6,-6) -- (-8.6,-8) -- (-8.6,-6) -- (-8,-4) -- (-8,-6) -- (-8,-8) -- (-8,-6) -- (-7.8,-8) -- (-8,-6) -- (-8,-4) -- (-6,-2) -- (-6,-4) -- (-6,-6) -- (-6,-8) -- (-6,-6) -- (-5.8, -8) -- (-6,-6) --(-6,-4) -- (-5.4,-6) -- (-5.2,-8) -- (-5.4,-6) -- (-5.6,-8) -- (-5.4,-6) -- (-6,-4);
            \draw[line width=0.5mm,green!50!black, opacity = 1, cap=round, ->] (-0.3,0.3) to (-5.7,-1.5);
            \node[rotate = 20] at (-3,-0.2) {{\color{green!50!black}$\texttt{root}$}};
			\draw [<-, blue!50!black, thick, shorten <= 0.3cm] (0,2) to (0,0);
			\draw [<-, blue!50!black, thick, shorten <= 0.3cm] (0,0) to (-6,-2);
			\draw [<-, blue!50!black, thick, shorten <= 0.3cm] (0,0) to (0,-2);
			\draw [<-, blue!50!black, thick, shorten <= 0.3cm] (0,0) to (6,-2);
			
			\draw [<-, blue!50!black, thick, shorten <= 0.3cm] (-6,-2) to (-8,-4);
			\draw [<-, blue!50!black, thick, shorten <= 0.3cm] (-6,-2) to (-6,-4);
			\draw [<-, blue!50!black, thick, shorten <= 0.3cm] (-6,-2) to (-4,-4);
			
			\draw [<-, blue!50!black, thick, shorten <= 0.3cm] (0,-2) to (-2,-4);
			\draw [<-, blue!50!black, thick, shorten <= 0.3cm] (0,-2) to (0,-4);
			\draw [<-, blue!50!black, thick, shorten <= 0.3cm] (0,-2) to (2,-4);
			
			\draw [<-, blue!50!black, thick, shorten <= 0.3cm] (6,-2) to (4,-4);
			\draw [<-, blue!50!black, thick, shorten <= 0.3cm] (6,-2) to (6,-4);
			\draw [<-, blue!50!black, thick, shorten <= 0.3cm] (6,-2) to (8,-4);
			
			\foreach \i in {-8,-6,-4,-2,0,2,4,6,8}{
				\foreach \j in {-0.6,0,0.6}{
					\draw [<-, blue!50!black, thick, shorten <= 0.3cm] (\i,-4) to (\i-\j,-6);
					\draw [-, blue!50!black, thick, shorten <= 0.2cm] (\i-\j,-6) to (\i-\j-0.2,-8);
					\draw [-, blue!50!black, thick, shorten <= 0.2cm] (\i-\j,-6) to (\i-\j,-8);
					\draw [-, blue!50!black, thick, shorten <= 0.2cm] (\i-\j,-6) to (\i-\j+0.2,-8);
			}
				\draw[fill = white] (\i-0.6,-6) circle (0.2);
				\node at (\i-0.6,-6) {\scalebox{0.6}{$0$}};
				\draw[fill = white] (\i,-6) circle (0.2);
				\node at (\i,-6) {\scalebox{0.6}{$1$}};
				\draw[fill = white] (\i+0.6,-6) circle (0.2);
				\node at (\i+0.6,-6) {\scalebox{0.6}{$2$}};
			}

			\draw[fill = white] (0,0) circle (0.3);
			\node at (0,0) {$1$};
			\draw[fill = white] (-6,-2) circle (0.3);
			\node at (-6,-2) {$0$};
			\draw[fill = white] (0,-2) circle (0.3);
			\node at (0,-2) {$1$};
			\draw[fill = white] (6,-2) circle (0.3);
			\node at (6,-2) {$2$};
			\draw[fill = white] (-8,-4) circle (0.3);
			\node at (-8,-4) {$0$};
			\draw[fill = white] (-6,-4) circle (0.3);
			\node at (-6,-4) {$1$};
			\draw[fill = white] (-4,-4) circle (0.3);
			\node at (-4,-4) {$2$};
			\draw[fill = white] (-2,-4) circle (0.3);
			\node at (-2,-4) {$0$};
			\draw[fill = white] (0,-4) circle (0.3);
			\node at (0,-4) {$1$};
			\draw[fill = white] (2,-4) circle (0.3);
			\node at (2,-4) {$2$};
			\draw[fill = white] (4,-4) circle (0.3);
			\node at (4,-4) {$0$};
			\draw[fill = white] (6,-4) circle (0.3);
			\node at (6,-4) {$1$};
			\draw[fill = white] (8,-4) circle (0.3);
			\node at (8,-4) {$2$};

		\end{scope}
	\end{tikzpicture}
	\caption{The terniary structure induced by $\mathbf{T}$. The three incoming arrows represent $s_1(a),s_2(a),s_3(a)$, the outgoing arrow represents $p(a)$ and the number $c(a)$. We mark in red the binary tree associated to the topmost element according to $\gamma$. We can see that the choice in the definition of $\gamma$ forces the binary trees to be pairwise disjoint.}
	\label{fig:arbolitos}
\end{figure}

We end this section with a verification that $\gamma$ satisfies both conditions of~\Cref{lem:arbolitos}.

\begin{proof}[Proof of~\Cref{lem:arbolitos}]
	Let $\tau \in \mathbf{T}$, $w \in \{\symb{0},\symb{1}  \}^*$ and $g \in N$. We check condition (1) by induction on the size of $w$. Note first that $(g^{-1}\tau)(1_{N}) = \tau(g)$, this immediately implies that $\gamma(\tau,\epsilon,g) = g\gamma(g^{-1}\tau,\epsilon,1_{N})$. Now write $w = ub$ for some $b \in \{\symb{0},\symb{1}\}$ and suppose that $\gamma(\tau,u,g) = g\gamma(g^{-1}\tau,u,1_{N})$. We have that \[\tau(\gamma(\tau,u,g)) = \tau(g\gamma(g^{-1}\tau,u,1_{N})) = g^{-1}\tau (\gamma(g^{-1}\tau,u,1_{N})).\]
	From this and the definition of $\gamma$ it follows that $\gamma(\tau,w,g) = g\gamma(g^{-1}\tau,w,1_{N})$.
	
	Let us now show condition (2). Let $g,g' \in N$ and $w,w' \in \{\symb{0},\symb{1}  \}^*$ such that $\gamma(\tau,w,g) = \gamma(\tau,w',g')$. If neither $w$ nor $w'$ is the empty word, we write $w = ub$ and $w'=u'b'$ for $b,b' \in \{\symb{0},\symb{1}\}$. From this assumption and the relation $\gamma(\tau,u,g) = \gamma(\tau,w,g)p(\tau(\gamma(\tau,w,g)))$, we obtain that $\gamma(\tau,u,g) = \gamma(\tau,u',g')$. Iterating this procedure we may assume without loss of generality that $w' = \epsilon$.
	
	Suppose then that we have $\gamma(\tau,w,g) = \gamma(\tau,\epsilon,g')$ for some $w \in \{\symb{0},\symb{1}  \}^*$. Let $h = \gamma(\tau,\epsilon,g')$ and $a = \tau(h)$. By definition of $\gamma(\tau,\epsilon,g')$ we have that \[ c(\tau(h)) +1= c(\tau(hp(\tau(h))) \bmod{3}. \]
	If $w$ were nonempty, we would instead have that either \[ c(\tau(h)) = c(\tau(hp(\tau(h)))  \mbox{ or } c(\tau(h))+2 = c(\tau(hp(\tau(h))) \bmod{3}. \]
	From where we conclude that $w = \epsilon$ as well. Finally, if $h = \gamma(\tau,\epsilon,g) = \gamma(\tau,\epsilon,g')$, we can multiply on the right by $p(h)$ to conclude that $g = g'$. This shows condition (2).
\end{proof}

\subsection{Reduction of the main result to the case $H \times \{0,1\}^*$} Now we shall employ~\Cref{lem:arbolitos} to reduce the proof of~\Cref{thm:main_theorem} to the construction of a configuration space in $H \times \{0,1\}^*$. For that we will need to provide an ad-hoc extension of the definition of subshift of finite type for the direct product $M = H \times \{0,1\}^*$ of some group $H$ with the rooted free monoid $\{0,1\}^{*}$.

Let $A$ be a finite set. For some finite subset $F\Subset M$, a function $p \colon F \to A$ is called a \define{pattern} and $F = \operatorname{supp}(p)$ is called its \define{support}.

\begin{definition}\label{def:rootedSFT}
	Let $H$ be a group and consider the direct product $M = H \times \{0,1\}^*$. Let $A$ be a finite set. Given $A_0 \subset A$ and a finite set of forbidden patterns $\mathcal{F}$, the \define{rooted SFT} induced by $A_0$ and $\mathcal{F}$ is the subset $X\subset A^{M}$ which satisfies that $x \in X$ if and only if 
	\begin{enumerate}
		\item For every $g \in M$ and $p \in \mathcal{F}$ there exists $m \in \operatorname{supp}(p)$ such that $x(gm) \neq p(m)$.
		\item For every $h \in H$, we have that $x(h,\epsilon) \in A_0$.
	\end{enumerate}
\end{definition}

The first condition states that forbidden patterns cannot occur anywhere in $M$. The second condition restricts the symbols that may occur in any coordinate in which the second coordinate is the root (the empty word) in $\{0,1\}^*$.

As usual with SFTs, we can describe forbidden patterns implicitly through invariant rules. For instance, if we say: ``let $n \in M$, for every $m \in M$ if $x(m) = a$ then $x(mn)=b$'', this can be encoded as the set of forbidden patterns $p$ with support $\{1_M,n\}$ such that $p(1_M)=a$ and $p(n)\neq b$.

\begin{definition}\label{def:subfactors}
	Let $A,B$ be finite sets, $X\subset A^H$ and $Z \subset B^{M}$. We say that $Z$ \define{root-factors onto} $X$ if there is a map $\Phi\colon B \to A$ such that if we let $\phi\colon Z \to A^{H}$ be given by $\phi(z)(h) = \Phi(z(h,\epsilon))$ for every $h \in H$, then $\phi(Z) = X$.
\end{definition}

Next we state the lemma that will enable us to restrict our attention to the monoid $H \times \{0,1\}^*$.

\begin{lemma}\label{lem:reduction}
	Let $H$ be a finitely generated group and suppose that $Z$ is a rooted SFT on $M= H \times \{0,1\}^*$ which root-factors onto an $H$-subshift $X$. Then for every non-amenable group $N$ the free extension of $X$ to $G = H \times N$ is a sofic subshift. 
\end{lemma}

\begin{proof}
Fix a nonamenable group $N$ and let $Z\subset B^M$ be the rooted SFT which root-factors onto $X$. Let $B_0 \subset B$ be its seed alphabet, $\mathcal{F}$ a set of forbidden patterns which defines it and $\Phi \colon B \to A$ the associated root-factor map.

Let $\mathbf{T}\subset A^N$ be the binary tree shift on $N$. Let $\texttt{root}, \texttt{son}_0, \texttt{son}_1$ and $\gamma \colon \mathbf{T}\times \{0,1\}^* \times N \to N$ be the continuous maps from~\Cref{lem:arbolitos}. Consider its trivial extension $\widetilde{\mathbf{T}}$ to $G = H \times N$
\[ \widetilde{\mathbf{T}} = \{ x\in A^{H \times N} : x|_{\{1_H\} \times N} \in \mathbf{T} \mbox{ and } x(h,n) = x(1_H,n) \mbox{ for every } h \in H   \}.  \]
As $H$ is finitely generated, it follows that $\widetilde{\mathbf{T}}$ is a $G$-SFT. Also note that all of the previous maps extend naturally to $\widetilde{\mathbf{T}}$. For instance, $\gamma$ extends to a continuous map $\widetilde{\gamma} \colon \widetilde{\mathbf{T}}\times \{0,1\}^* \times N \to N$ by letting $\widetilde{\gamma}(\widetilde{\tau},w,h) = \gamma(\tau,w,h)$ with $\widetilde{\tau}|_{\{1_H\}\times N} = \tau$. In order to reduce the encumbrance of notation, we shall keep the original names of these maps. 

Let us briefly explain the intuition of the remainder of the proof. We wish to construct a subshift of finite type on the product space $\widetilde{\mathbf{T}} \times B^G$ in the following way: the configuration $\tau \in \widetilde{\mathbf{T}}$ will associate to every $n \in N$ an infinite binary tree, and we will use the underlying rules to associate a set of forbidden patterns which force every induced copy of $H \times \{0,1\}^*$ to contain a copy of $Z$. However, notice that the maps $\gamma_{\tau} \colon \{0,1\}^* \times N \to N$ are not necessarily surjective, and thus there are portions of $N$ which carry either infinite unrooted binary trees, or degenerate trees (when the underlying component of the $3$-to-$1$ map ends in a cycle). In order to make sure we are able to correctly produce a non-empty subshift, we shall introduce an extra symbol $\ast$ which will serve the purpose of filling those degenerate portions of $N$.

Now we proceed formally, let $\ast \notin B$ and let $C = B \cup \{\ast\}$. We define the subshift $Y\subset \widetilde{\mathbf{T}} \times C^G$ as the set of configurations $(\tau,y)$ in $\widetilde{\mathbf{T}} \times C^G$ which satisfy the following extra rules. For every $(h,n)\in H \times N$, we have
\begin{enumerate}
	\item \define{root condition}: if we let $r = n\cdot \texttt{root}((h,n)^{-1}\tau)$ then $y(h,r) \in B_0$. 
	\item \define{pattern condition} for every $p \in \mathcal{F}$ there exists $(g,w) \in \operatorname{supp}(p)$ such that \[y(hg,\gamma(\tau,w,n)) \neq p(g,w).\]
	\item \define{extra symbol condition} let $n_0 = n\cdot \texttt{son}_0((h,n)^{-1}\tau)$ and $n_1 = n\cdot \texttt{son}_1((h,n)^{-1}\tau)$. We have that $y(h,n) = \ast \iff y(h,n_0)=\ast$ and that $y(h,n) = \ast \iff y(h,n_1)=\ast$.
\end{enumerate}

As the maps $\texttt{root}, \texttt{son}_0, \texttt{son}_1$ and $\gamma$ are continuous, they depend only on finitely many coordinates of $\tau$. Furthermore, as $\mathcal{F}$ is finite, there are only finitely many $w$ in the second condition. This means that the three rules above are local and can be implement with finitely many forbidden patterns. This shows that $Y$ is a $G$-SFT.

We now define a topological factor map $\psi$ from $Y$ to the free extension of $X$ as follows, for every $(\tau,y)\in Y$ we let \[ \psi(\tau,y)(h,n) = \Phi( y(h, \gamma(\tau,\epsilon,n))).  \] 

From the fact that $n \cdot \texttt{root}((h,n)^{-1}\tau) = \gamma(\tau,\epsilon,n)$ it is clear that $\psi$ is continuous. From the identity $\gamma(\tau,\epsilon,n) = n\gamma(n^{-1}\tau,\epsilon,1_N)$ it follows that $\psi$ is $G$-equivariant. 

Recall that the surjective map $\phi \colon Z \to X$ is given by $\phi(z)(h) = \Phi(z(h,\epsilon))$. Fix $(\tau,y)\in Y$ and for each $n \in N$ let $z_n \in C^{M}$ be given by \[z_n(h,w) = y(h,\gamma(\tau,w,n))\]

By the root condition we have that $z_n(h,\epsilon)\in A_0$. In particular, as the symbol $\ast$ is not in $A_0$, using the extra symbol condition recursively we obtain that for every $w \in \{0,1\}^*$ it holds that $z_n(h,w) \neq \ast$ and thus in fact we have $z_n \in B^M$. Finally, the pattern condition ensures that no forbidden patterns occur in $z_n$ and thus $z_n \in Z$. It follows that \[\psi(\tau,y)(1_H,n) = \Phi(y(h,\gamma(\tau,\epsilon,n))) = \Phi(z_n(h,\epsilon)) = \phi(z_n)(h).\]

From where we obtain that $\psi$ is well defined, that is, the range of $\psi$ is contained in the free extension of $X$. In order to conclude our proof we just need to show that $\psi$ is surjective.

Let $\widetilde{x}$ be in the free extension of $X$ to $H \times N$. For each $n \in N$, let $x_n\in X$ be given by $x_n(h) = \widetilde{x}(h,n)$. As $\phi \colon Z \to X$ is surjective, for each $n$ we may choose and fix some $z_n \in Z$ such that $\phi(z_n) = x_n$.

As $N$ is nonamenable, the subshift $\widetilde{\mathbf{T}}$ is nonempty (for a suitable choice of $K$ in its definition) and thus there is some $\tau \in \widetilde{\mathbf{T}}$. Partition the group $N$ into $N = R_{\tau} \cup U_{\tau}$ where $m \in R_{\tau}$ if and only if there is $n \in N$ and $w \in \{0,1\}^*$ such that $m = \gamma(\tau,w,n)$. The names $R_{\tau}$ and $U_{\tau}$ stand for ``reachable'' and ``unreachable'' respectively. Notice that by~\Cref{lem:arbolitos} if $m = \gamma(\tau,w,n) \in R_{\tau}$, then the choice of $n$ and $w$ is unique.

Define $y \in C^G$ as follows. For every $(h,m) \in H \times N$, \[ y(h,m) = \begin{cases}
	z_n(h,w) & \mbox{ if } m = \gamma(\tau,w,n) \in R_{\tau} \\
	\ast & \mbox{ if } m \in U_{\tau}. 
\end{cases}    \]

A direct verification shows that $(\tau,y)\in Y$ and that $\psi(\tau,y) = \widetilde{x}$ and thus $\psi$ is surjective.

As $Y$ is a $G$-SFT and $\psi$ is a continuous, $G$-equivariant surjective map onto the free extension of $X$, it follows that said extension is a $G$-sofic subshift.\end{proof}

\section{Soficity of free extensions}\label{sec:mainthm}

The purpose of this section is to show~\Cref{thm:main_theorem}, namely, that for every nonamenable group $N$, every finitely generated group $H$ with decidable word problem, and every effectively closed $H$-subshift $X$, the free extension of $X$ to $H \times N$ is a sofic subshift.

The proof in the case where $H$ is finite is trivial. In view of~\Cref{lem:reduction}, it suffices to show that if $H$ is an infinite group with decidable word problem then for every effectively closed by patterns $H$-subshift $X$ there is a rooted SFT $Z$ on $H \times \{0,1\}^*$ which root-factors onto $X$. We remark that as the word problem of $H$ is decidable, effectively closed by patterns subshifts are effectively closed actions, and thus we will simply say ``effectively closed subshift''. For the remainder of the section we fix a finitely generated group $H$ with decidable word problem and an effectively closed $H$-subshift $X \subset A^{H}$. We also write $M = H \times \{0,1\}^*$.

Let us begin by explaining in very informal terms our construction of the rooted SFT $Z$. We will describe $Z$ as a subset of the product of five other structures which we will refer to as ``layers''. The alphabet layer $\boldsymbol{A}$, the directions layer $\boldsymbol{D}$, the branching layer $\boldsymbol{B}$, the computation layer $\boldsymbol{C}$ and the tentacle layer $\boldsymbol{T}$.

First we consider the full $H$-shift $A^H$ and extend it trivially to $M$, this is the alphabet layer $\boldsymbol{A}$. Second, we construct an SFT on $H$ on which every configuration induces a partition of $H$ into local copies of $\ZZ$ or one of its quotients. A theorem of Seward will ensure that there exists at least one configuration in which no proper quotients of $\ZZ$ occur. The directions layer $\boldsymbol{D}$ is the trivial extension of this space to $M$.

In the branching layer $\boldsymbol{B}$ we produce a structure suitable for parallel computation. Namely, we shall use the information from $\boldsymbol{D}$ and in each induced copy of $\ZZ \times \{0,1\}^*$ (or $\ZZ/n\ZZ \times \{0,1\}^*$) we will seed symbols in each position $(n,\epsilon)$ which will act as computation tapes of size $1$. Each time we advance on the tree, each computation nondeterministically chooses to move to one branch of the tree and extend its computation tape to the right by one place. As the number of possible branchings grows exponentially, there exist choices which enable each position to extend its own rooted computation tape to the right without ever colliding with the other tapes. This will also eliminate the possibility of witnessing any proper quotients of $\ZZ$ induced by $\boldsymbol{D}$. See~\Cref{fig:branchescomputation} for an illustration of this construction.

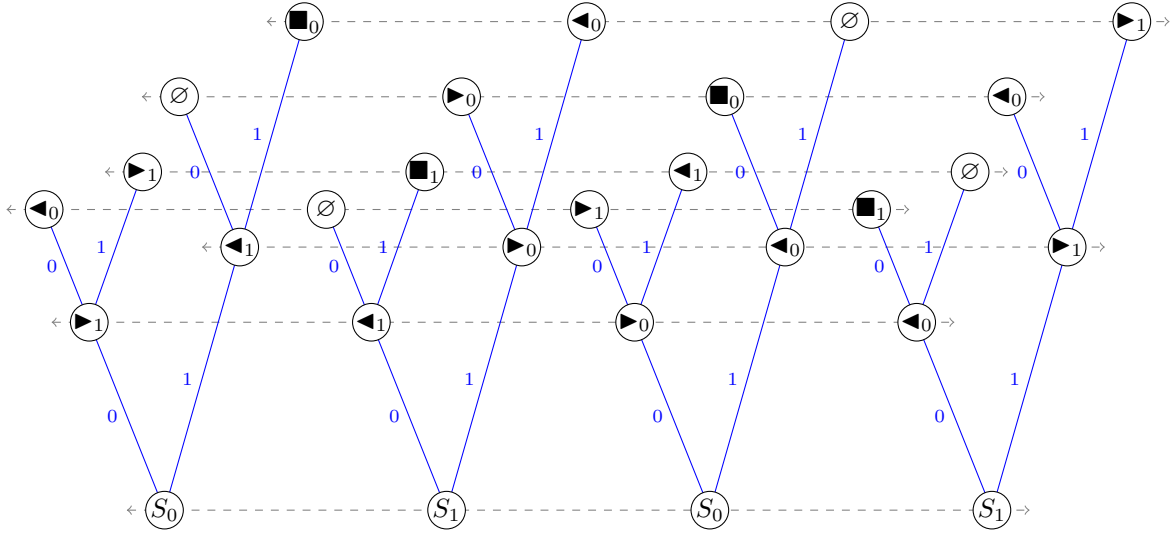
\begin{figure}[ht!]
	\begin{tikzpicture}
		\draw[<->, black!50, dashed] (-6,0) to (6,0);
		\draw[<->, black!50, dashed] (-7,2.5) to (5,2.5);
		\draw[<->, black!50, dashed] (-5,3.5) to (7,3.5);
		\draw[<->, black!50, dashed] (-7.6,4) to (4.4,4);
		\draw[<->, black!50, dashed] (-6.29,4.5) to (5.71,4.5);
		\draw[<->, black!50, dashed] (-5.8,5.5) to (6.2,5.5);
		\draw[<->, black!50, dashed] (-4.14,6.5) to (7.86,6.5);
		\begin{scope}[shift = {(-5.5,0)}]
			\draw [blue] (0,0) -- node [midway, left] {\scriptsize $0$} (-1,2.5) ;
			\draw [blue] (0,0) -- node [midway, left] {\scriptsize$1$} (1,3.5) ;
			\draw [blue] (-1,2.5) -- node [midway, left] {\scriptsize$0$} (-1.6,4);
			\draw [blue] (-1,2.5) -- node [midway, left] {\scriptsize$1$} (-0.29,4.5);
			\draw [blue] (1,3.5) -- node [midway, left] {\scriptsize$0$} (0.2,5.5);
			\draw [blue] (1,3.5) -- node [midway, left] {\scriptsize$1$} (1.86,6.5);
			\draw [fill = white ] (0,0) circle (0.25) node {$\texttt{S}_0$};
			\draw [fill = white ] (-1,2.5) circle (0.25) node {$\blacktriangleright_1$};
			\draw [fill = white ] (1,3.5) circle (0.25) node {$\blacktriangleleft_1$}; 
			\draw [fill = white ] (-1.6,4) circle (0.25) node {$\blacktriangleleft_0$}; 
			\draw [fill = white ] (-0.29,4.5) circle (0.25) node {$\blacktriangleright_1$}; 
			\draw [fill = white ] (0.2,5.5) circle (0.25) node {$\varnothing$}; 
			\draw [fill = white ] (1.86,6.5) circle (0.25) node {$\blacksquare_0$}; 
		\end{scope}
		\begin{scope}[shift = {(-1.75,0)}]
			\draw [blue] (0,0) -- node [midway, left] {\scriptsize $0$} (-1,2.5) ;
			\draw [blue] (0,0) -- node [midway, left] {\scriptsize$1$} (1,3.5) ;
			\draw [blue] (-1,2.5) -- node [midway, left] {\scriptsize$0$} (-1.6,4);
			\draw [blue] (-1,2.5) -- node [midway, left] {\scriptsize$1$} (-0.29,4.5);
			\draw [blue] (1,3.5) -- node [midway, left] {\scriptsize$0$} (0.2,5.5);
			\draw [blue] (1,3.5) -- node [midway, left] {\scriptsize$1$} (1.86,6.5);
			\draw [fill = white ] (0,0) circle (0.25) node {$\texttt{S}_1$};
			\draw [fill = white ] (-1,2.5) circle (0.25) node {$\blacktriangleleft_1$};
			\draw [fill = white ] (1,3.5) circle (0.25) node {$\blacktriangleright_0$};
			\draw [fill = white ] (-1.6,4) circle (0.25) node {$\varnothing$}; 
			\draw [fill = white ] (-0.29,4.5) circle (0.25) node {$\blacksquare_1$}; 
			\draw [fill = white ] (0.2,5.5) circle (0.25) node {$\blacktriangleright_0$};
			\draw [fill = white ] (1.86,6.5) circle (0.25) node {$\blacktriangleleft_0$}; 
		\end{scope}
		\begin{scope}[shift = {(1.75,0)}]
			\draw [blue] (0,0) -- node [midway, left] {\scriptsize $0$} (-1,2.5) ;
			\draw [blue] (0,0) -- node [midway, left] {\scriptsize$1$} (1,3.5) ;
			\draw [blue] (-1,2.5) -- node [midway, left] {\scriptsize$0$} (-1.6,4);
			\draw [blue] (-1,2.5) -- node [midway, left] {\scriptsize$1$} (-0.29,4.5);
			\draw [blue] (1,3.5) -- node [midway, left] {\scriptsize$0$} (0.2,5.5);
			\draw [blue] (1,3.5) -- node [midway, left] {\scriptsize$1$} (1.86,6.5);
			\draw [fill = white ] (0,0) circle (0.25) node {$\texttt{S}_0$};
			\draw [fill = white ] (-1,2.5) circle (0.25) node {$\blacktriangleright_0$};
			\draw [fill = white ] (1,3.5) circle (0.25) node {$\blacktriangleleft_0$};
			\draw [fill = white ] (-1.6,4) circle (0.25) node {$\blacktriangleright_1$}; 
			\draw [fill = white ] (-0.29,4.5) circle (0.25) node {$\blacktriangleleft_1$}; 
			\draw [fill = white ] (0.2,5.5) circle (0.25) node {$\blacksquare_0$}; 
			\draw [fill = white ] (1.86,6.5) circle (0.25) node {$\varnothing$}; 
		\end{scope}
		\begin{scope}[shift = {(5.5,0)}]
			\draw [blue] (0,0) -- node [midway, left] {\scriptsize $0$} (-1,2.5) ;
			\draw [blue] (0,0) -- node [midway, left] {\scriptsize$1$} (1,3.5) ;
			\draw [blue] (-1,2.5) -- node [midway, left] {\scriptsize$0$} (-1.6,4);
			\draw [blue] (-1,2.5) -- node [midway, left] {\scriptsize$1$} (-0.29,4.5);
			\draw [blue] (1,3.5) -- node [midway, left] {\scriptsize$0$} (0.2,5.5);
			\draw [blue] (1,3.5) -- node [midway, left] {\scriptsize$1$} (1.86,6.5);
			\draw [fill = white ] (0,0) circle (0.25) node {$\texttt{S}_1$};
			\draw [fill = white ] (-1,2.5) circle (0.25) node {$\blacktriangleleft_0$};
			\draw [fill = white ] (1,3.5) circle (0.25) node {$\blacktriangleright_1$}; 
			\draw [fill = white ] (-1.6,4) circle (0.25) node {$\blacksquare_1$}; 
			\draw [fill = white ] (-0.29,4.5) circle (0.25) node {$\varnothing$}; 
			\draw [fill = white ] (0.2,5.5) circle (0.25) node {$\blacktriangleleft_0$}; 
			\draw [fill = white ] (1.86,6.5) circle (0.25) node {$\blacktriangleright_1$}; 
		\end{scope}

	\end{tikzpicture}
	\caption{The computation branches in $\ZZ \times \{0,1\}^*$. Each computation zone is of the form $\blacktriangleright(\blacksquare)^*\blacktriangleleft$. The $\texttt{S}_i$ symbols are seeds and the $\varnothing$ symbols are placeholders for unused space. The numbers indicate the branch in which the computation will continue, note that the numbers are constant inside every computation zone.}
	\label{fig:branchescomputation}
\end{figure}

Next we use the information from the branching layer $\boldsymbol{B}$ to construct a computation layer $\boldsymbol{C}$ whose purpose is to simulate a Turing machine rooted on every position of $H$ that enumerates forbidden pattern codings for $X$. As soon as a forbidden pattern coding is produced, each word in its description will be replaced by a geodesic (using the algorithm for the word problem of the group). Next, a ``search'' subroutine will start, in which the computation layer $\boldsymbol{C}$ will locally interact with the tentacle layer $\boldsymbol{T}$; This tentacle layer grows ``tentacles'' from the leftmost position of every computation layer which follow the geodesic path and check if a symbol occurs in some position in $H$ and bring the information back to the computation layer. We will show that there is a choice of branchings such that no tentacles attached to distinct computations collide. The reason behind turning the words into geodesics is to ensure that the tentacles do not collide with themselves.

If either two computation tapes collide, two tentacles collide, or a forbidden pattern coding is found to occur in $A^H$, then we force the occurrence of a forbidden pattern. This will ensure that the configurations of $A^H$ which belong on $X$ are exactly those that occur as projections to the $A^H$-layer of configurations of $Z$, thus showing that $Z$ root-factors onto $X$.

\subsection{The alphabet layer}

Recall that $M = H \times \{0,1\}^*$. We define the alphabet layer $\boldsymbol{A}$ as the trivial extension to $M$ of the full $H$-shift on $H$, that is, \[ \boldsymbol{A} = \{ \alpha \in A^{M} : \alpha(h,w) = \alpha(h,\epsilon) \mbox{ for every } h \in H, w \in \{0,1\}^* \}.   \]

\subsection{The directions layer}

A bijection $T \colon H \to H$ is a \define{translation-like action} of $\ZZ$ on $G$ if it is free (for every $h \in H$, $T^k(h)= h$ implies $k = 0$) and there is $K\Subset H$ such that $h^{-1}T(h) \in K$ for every $h \in H$. A result of Seward~\cite[Theorem 1.4]{Seward2014} (see also~\cite{Carrascovargas2023geometric} for a computable version) shows that every infinite and finitely generated group $H$ admits a translation-like action of $\ZZ$. 

For the remainder of this section, fix a finite symmetric set $S$ of generators of $H$ which contains the identity and such that there exists a translation-like action $T\colon H \to H$ with $h^{-1}T(h) \in S$. Consider the alphabet $A_{\boldsymbol{D}}=S \times S$. For $a = (s,s')\in S \times S$ denote $\ell(a) = s$, $r(a)=s'$ where $\ell$ and $r$ stand for ``left'' and ``right''.

Consider the $H$-SFT $D$ as the space of all configurations $d \in (A_{\boldsymbol{D}})^{H}$ such that for any $h \in H$, \begin{enumerate}
	\item If $s = r(d(h))$ and $u = \ell(d(hs))$ then $s = u^{-1}$.
	\item If $s = \ell(d(h))$ and $u = r(d(hs))$ then $s = u^{-1}$.
\end{enumerate}

Both rules are local and thus $D$ is an $H$-SFT. Notice that each $d \in D$ induces a map $T_d \colon H \to H$ by letting $T_d(h) = h\cdot r(d(h))$ and the rules imply that $T_d$ is a bijection (its inverse is given by $T_d^{-1}(h) = h \cdot \ell(d(h))$). Remark that for every $h \in H$ we have $h^{-1}T_d(h) = r(d(h)) \in S$. Conversely, every bijection $T\colon H \to H$ with $h^{-1}T(h) \in S$ for every $h \in H$ induces a configuration $d_T \in D$ by letting $\ell(d_T(h)) = h^{-1}T^{-1}(h)$ and $r(d_T(h)) = h^{-1}T(h)$.

Let ${\boldsymbol{D}}$ be the trivial extension of ${D}$ to $M = H \times \{0,1\}^*$, that is \[ {\boldsymbol{D}} = \{ \delta \in (A_{\boldsymbol{D}})^{M} : \delta|_{H \times \{\epsilon\}} \in {D} \mbox{ and } \delta(h,w) = \delta(h,\epsilon) \mbox{ for every } h \in H, w \in \{0,1\}^*    \}.  \]

Clearly $\boldsymbol{D}$ is a rooted SFT (with no root condition at all). We call $\boldsymbol{D}$ the direction layer. Given $\delta \in \boldsymbol{D}$ and $h \in H$, we will use the notations $L_{\delta}(h) = h\ell(\delta(h,\epsilon))$ and $R_{\delta}(h) = hr(\delta(h,\epsilon))$ to denote the elements at the left and right of $h$ induced by $\delta$.

\subsection{The branching layer}

Now we will describe how to produce a branching structure of expanding computation zones which are seeded everywhere in $H$. Consider the alphabet \[A_{\boldsymbol{B}} = \{\varnothing \} \cup \left(\{ \texttt{S}, \blacktriangleleft, \blacksquare, \blacktriangleright\}  \times \{0,1\}  \right).  \]

Also consider the seed subalphabet as \[  A'_{\boldsymbol{B}} = \{\texttt{S}\}\times \{0,1\}.   \]

Unlike the layers so far (which were independent from each other), on the branching layer the local rules will depend upon information locally encoded in the configuration $\delta \in \boldsymbol{D}$. We write the rules with $\delta$ as a parameter, which are then straightforward to translate into forbidden patterns on the product alphabet.

We define $\boldsymbol{B}$ as the set of all $\beta \in (A_{\boldsymbol{B}})^M$ which satisfy the following rules for every $h \in H$ and $w \in \{0,1\}^*$,

\begin{enumerate}
	\item \textbf{seed rule:} we have that $\beta(h,\epsilon)\in A'_{\boldsymbol{B}}$.
	\item \textbf{seed growth:} for each $b \in \{0,1\}$, if $\beta(h,w) = (\texttt{S},b)$ then there is $c \in \{0,1\}$ such that $\beta(h,wb) = (\blacktriangleright,c)$ and $\beta(R_{\delta}(h),wb) = (\blacktriangleleft,c)$
	\item \textbf{horizontal consistency rules:} for each $b \in \{0,1\}$,
	\begin{enumerate}
		\item if $\beta(h,w) = (\blacktriangleright,b)$ then $\beta(R_{\delta}(h),w) \in \{ (\blacksquare,b), (\blacktriangleleft,b) \}$.
		\item if $\beta(h,w) = (\blacksquare,b)$ then $\beta(R_{\delta}(h),w) \in \{ (\blacksquare,b), (\blacktriangleleft,b) \}$ and $\beta(L_{\delta}(h),w) \in \{ (\blacktriangleright,b), (\blacksquare,b) \}$.
		\item if $\beta(h,w) = (\blacktriangleleft,b)$ then $\beta(L_{\delta}(h),w) \in \{ (\blacktriangleright,b), (\blacksquare,b) \}$.
	\end{enumerate}	
	\item \textbf{branching rules:} for each $b \in \{0,1\}$,
	\begin{enumerate}
		\item if $\beta(h,w) = (\blacktriangleright,b)$ then there is $c \in \{0,1\}$ such that $\beta(h,wb) = (\blacktriangleright,c)$.
		\item if $\beta(h,w) = (\blacksquare,b)$ then there is $c \in \{0,1\}$ such that $\beta(h,wb) = (\blacksquare,c)$.
		\item if $\beta(h,w) = (\blacktriangleleft,b)$ then there is $c \in \{0,1\}$ such that $\beta(h,wb) = (\blacksquare,c)$ and $\beta(R_{\delta}(h),wb) = (\blacktriangleleft,c)$.
	\end{enumerate}	
	\item \textbf{Computation symbols come from seeds:} If $|w|\geq 1$, write $w = ub$ for some $b \in \{0,1\}$. Unless $\beta(h,u) \in \{ (\blacktriangleright,b), (\blacksquare,b),(\blacktriangleleft,b), (\texttt{S},b)\}$ or $\beta(L_{\delta}(h),u) \in \{ (\blacktriangleleft,b), (\texttt{S},b)\}$, we have $\beta(h,w)=\varnothing$.
\end{enumerate}

Let us now explain the meaning of the alphabet and the rules. The symbols $\blacktriangleright,\blacksquare$ and $\blacktriangleleft$ encode ``computation zones''. Every computation zone is read from left to right in an induced copy of $\ZZ$ by $\delta$ and is of the form $\blacktriangleright(\blacksquare)^*\blacktriangleleft$. The symbols $\{0,1\}$ are branching bits, and encode into which branch of the infinite binary tree $\{0,1\}^*$ the tape should copy itself and extend to the right. The symbol $\texttt{S}$ is a seed which ensures that from that position a computation zone will start growing. Finally, the symbol $\varnothing$ is a placeholder symbol that can go anywhere (as long as the rules don't force another symbol in that place). 

The seed rule states that the seed symbol $\texttt{S}$ should occur at every position in $H$ with an arbitrary branch number $b \in \{0,1\}$. The seed growth rule forces the seed to move into the branch $b$ and spawn a computation zone of size $2$ of the form $\blacktriangleright\blacktriangleleft$. The horizontal consistency rules ensure that every computation zone shares the same branch bit $b \in \{0,1\}$ and that they are always of the form $\blacktriangleright(\blacksquare)^*\blacktriangleleft$. The branching rules state that every computation zone marked by the branching bit $b$, whose leftmost border is in position $(h,w)\in M$, must copy itself into $(h,wb)$ and extend itself to the right by one position. Finally, the last rule states that nonempty symbols can only occur when a computation zone in a previous level extends, that is, everything else is necessarily filled with the placeholder symbol $\varnothing$.

Clearly, all of the above rules are either seed rules or local rules, and thus the set of configurations $(\alpha,\delta,\beta) \in \boldsymbol{A}\times \boldsymbol{D} \times \boldsymbol{B}$ form a rooted SFT. 

The next remarks are not needed for our proof, but they do help to understand how this construction works so far.

\begin{remark}
	Although the seed rule says that the symbol $\texttt{S}$ occurs in every $(h,\epsilon)$ with an arbitrary branch number, the only possibility is that in every induced copy of $\ZZ$ (or a quotient) the branch numbers alternate between $0$ and $1$. Otherwise, the branching rule would enforce that two distinct symbols must appear at a single position. Strict alternation is not enforced on following steps, however: the tapes grow linearly but can separate at exponential rate, so we quickly start having many choices on the branch sequence.
\end{remark}

\begin{remark}
	For every $(\alpha,\delta,\beta) \in \boldsymbol{A}\times \boldsymbol{D} \times \boldsymbol{B}$, if we let $T_{\delta}\colon H \to H$ be the action $\ZZ \overset{T_{\delta}}{\curvearrowright} H$ given by $T_{\delta}(h) = hR_{\delta}(h)$ then $T_{\delta}$ is a translation-like action. Indeed, suppose there is $n >0$ such that $T_{\delta}^n(h) = h$. Define $b_0$ as the second coordinate of $\beta(h,\epsilon)$, and iteratively, for $i \in \{1,\dots,n\}$ let $b_i$ be the second coordinate of $\beta(h,b_0\dots b_{i-1})$. The seed growth rule implies that $\beta(h,b_0) = (\blacktriangleright,b_1)$, and iteratively, the branching rule a) forces that $\beta(h,b_0\dots b_i) = (\blacktriangleright,b_{i+1})$ for every $i \in \{1,\dots,n-1\}$. Similarly, the seed growth rule and the horizontal consistency rule implies that $\beta(T_{\delta}(h),b_0) = (\blacktriangleleft,b_1)$, and iteratively, the branching rule c) and the horizontal consistency rule force that $\beta(T^{i+1}_{\delta}(h),b_0\dots b_i) = (\blacktriangleleft,b_{i+1})$ for every $i \in \{1,\dots,n-1\}$. If we have $T^{n}(h) = h$, then \[(\blacktriangleleft,b_{n}) = \beta(h,b_{_0\dots b_{n-1}}) = (\blacktriangleright,b_{n})\]
	Which obviously cannot occur.
\end{remark}

\begin{remark}
	The rooted SFT $\boldsymbol{A}\times \boldsymbol{D} \times \boldsymbol{B}$ is non-empty. Our choice of generating set $S$ for $H$ ensures that there exists a translation-like action $T \colon H \to H$ with $h^{-1}T(h)\in S$ and thus $\delta$ defined by $\delta(h,w) = (h^{-1}T(h),h^{-1}T^{-1}(h))$ is an element of $\boldsymbol{D}$ with the property that $n \mapsto R^{n}_{\delta}(h)$ is injective for every $h \in H$. One way to construct a configuration $\beta \in \boldsymbol{B}$ compatible with $\delta$ is by assigning alternating branching values to the seeds in $H \times \{\epsilon\}$ in each induced copy of $\ZZ \times \{\epsilon\}$ and then iteratively doing the same in each computation zone in the following branches. A straightforward computation shows that this construction will ensure that there are $2^{|w|}-(|w|+1)$ unused spaces between each computation zone in a branch $w\in \{0,1\}^*$, which can be filled with the symbol $\varnothing$. In fact, it suffices to force alternating branching bits only on words with length $|w|=2^{k}-1$ for some $k \in \NN$.
\end{remark}

\subsection{The computation layer}

Recall that a Turing machine is a tuple $(Q,\Sigma,\delta,q_0,q_F)$ where $Q$ is a finite set of \define{states}, $\Sigma$ is a finite set called \define{alphabet}, $\delta \colon Q \times \Sigma \to Q \times \Sigma \times \{-1,0,1\}$ is a \define{transition function}, $q_0 \in Q$ is the \define{initial state} and $q_F\in Q$ is the \define{final state}. We shall also assume that there is a special \define{blank symbol} $\sqcup \in \Sigma$

Turing machines act (as monoids) on the space $\Sigma^{\ZZ}\times Q \times \ZZ$ by sending $(x,q,n)$ to $(y,r,m)$ where $\delta(q,x(n)) = (r,a,d)$, $m = n+d$, and $y \in \Sigma^{\ZZ}$ is such that $y(n)=a$ and $y_{\ZZ\setminus \{n\}} = x_{\ZZ \setminus \{n\}}$.

It is well-known that every Turing machine can be implemented in an equivalent way in a Turing machine which never uses the negative portion of the tape, that is, a Turing machine which acts on $\Sigma^{\NN}\times Q \times \NN$; and that their space-time diagrams can be implemented through Wang tilings. Here we give a slight adaptation of these classical constructions that match up well with the structure induced by the branching layer.

Fix a Turing machine $\mathcal{T}=(Q,\Sigma,\delta,q_0,q_F)$. We define the alphabet $A^{0}_{\mathcal{T}}$ as the set of squares with colored edges illustrated on~\Cref{fig:wangtiles}.

\begin{figure}[ht!]

\begin{tikzpicture}[scale =1.8]
		\begin{scope}[shift = {(0.75,1.5)}]
			\node at (0,-0.65) {\textbf{$\texttt{seed}$}};
			\seed
		\end{scope}
		
		\begin{scope}[shift = {(2.25,1.5)} ]
			\node at (0,-0.65) {\textbf{$\texttt{seed\_L}$}};
			\seedL
		\end{scope}
		
		\begin{scope}[shift = {(3.75,1.5)} ]
			\node at (0,-0.65) {\textbf{$\texttt{Left}$}};
			\lefttile
		\end{scope}
		\begin{scope}[shift={(0.75,0)}]
		\begin{scope}[shift = {(-1.5,0)} ]
			\node at (0,-0.65) {\textbf{$\texttt{seed\_R}$}};
			\seedR
		\end{scope}
		
		\begin{scope}[shift = {(0,0)} ]
			\node at (0,-0.65) {\textbf{$\texttt{Mid}$}};
			\midtile
		\end{scope}
		
		\begin{scope}[shift = {(1.5,0)} ]
			\node at (0,-0.65) {\textbf{$\texttt{Right}$}};
			\righttile
		\end{scope}
		
		\begin{scope}[shift = {(3,0)} ]
			\node at (0,-0.65) {\textbf{$\texttt{Origin}$}};
			\origintile
		\end{scope}
		
		\begin{scope}[shift = {(4.5,0)} ]
			\node at (0,-0.65) {\textbf{$\texttt{Blank}$}};
			\blanktile
		\end{scope}	
	\end{scope}
	\begin{scope}[shift ={(0,-1.5)}]
		
		\begin{scope}[shift = {(-1.5,0)} ]
			\tiletransmit{a}
		\end{scope}
		
		\begin{scope}[shift = {(0,0)} ]
			\tilestateup{s}{b}{s'}{b'}
		\end{scope}
		
		\begin{scope}[shift = {(1.5,0)} ]
			\tilestateleft{\ell}{c}{\ell'}{c'}
		\end{scope}
		
		\begin{scope}[shift = {(3,0)} ]
			\tilestateright{r}{d}{r'}{d'}
		\end{scope}
		
		\begin{scope}[shift = {(4.5,0)} ]
			\tilestateincomingleft{q}{a}
		\end{scope}
		
		\begin{scope}[shift = {(6,0)} ]
			\tilestateincomingright{q}{a}
		\end{scope}

	\end{scope}
\end{tikzpicture}

	\caption{The alphabet $A^{0}_{\mathcal{T}}$ associated to a Turing machine $\mathcal{T} = \{Q,\Sigma, \delta,q_0,q_F\}$. $\sqcup$ is the blank symbol, $q_0$ the starting state. For the bottom row tiles, $a \in \Sigma$ is an arbitrary symbol and $q \in Q$ is an arbitrary state. $(s,b),(\ell,c),(r,d) \in Q \times \Sigma$ are pairs such that $\delta(s,b)=(s',b',0)$, $\delta(\ell,c)=(\ell',c',-1)$ and $\delta(r,d)=(r',d',+1)$.}
	\label{fig:wangtiles}
\end{figure}
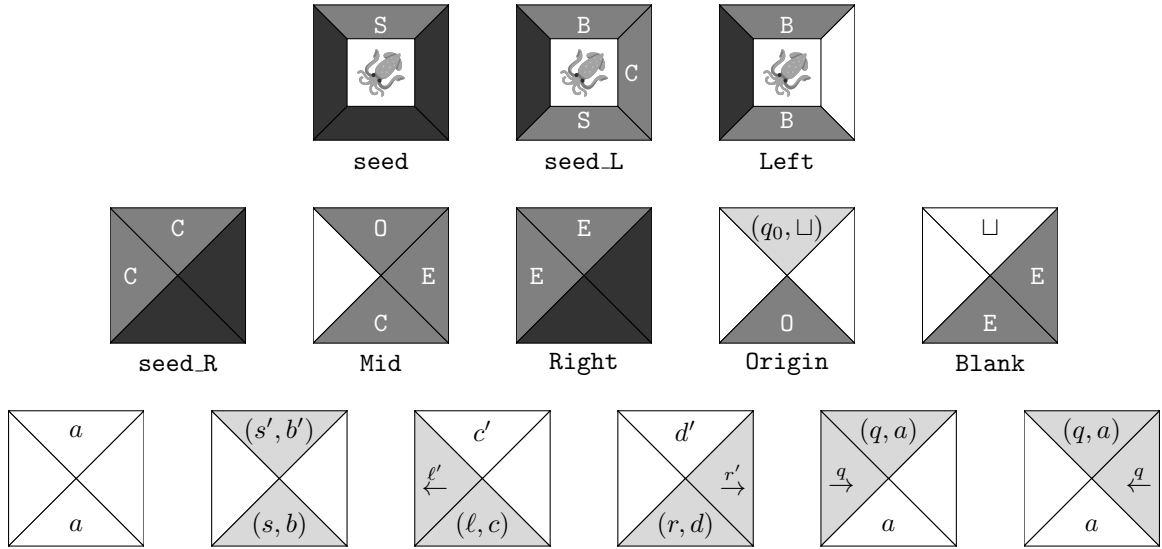

Notice that the first three tiles have a little squid drawn over them. This will be relevant for interactions with the tentacle layer $\boldsymbol{T}$, for now it may be ignored.

Let us illustrate how the computation tape will work. Recall that the branching layer induces on each position $(h,\epsilon)\in M$ a structure which we identify as a computation tape which extends to the right with every branching. We may identify each of these tapes with $\{ (n,m)\in \NN^2 : n \leq m \}$ and tile them with elements of $A^0_{\mathcal{T}}$ using as rule that the seed symbols on $\boldsymbol{B}$ must match with the seed tile, and that the symbols and colors in adjacent edges must match. See~\Cref{fig:machine_example}

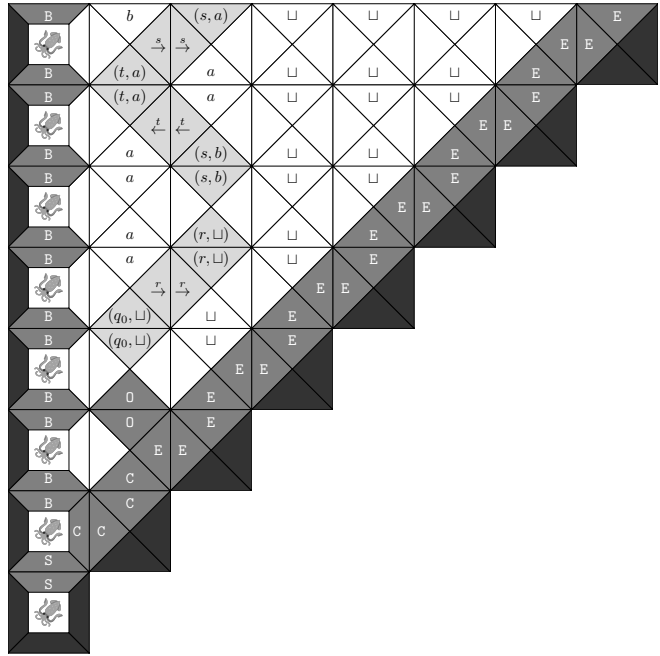
\begin{figure}[ht!]
	\scalebox{0.6}{
	\begin{tikzpicture}[scale = 1.8]
		\begin{scope}[shift = {(0,7)} ]
			\lefttile
		\end{scope}
		\begin{scope}[shift = {(1,7)} ]
			\tilestateright{t}{a}{s}{b}
		\end{scope}
		\begin{scope}[shift = {(2,7)} ]
			\tilestateincomingleft{s}{a}
		\end{scope}
		\begin{scope}[shift = {(3,7)} ]
			\tiletransmit{\sqcup}
		\end{scope}
		\begin{scope}[shift = {(4,7)} ]
			\tiletransmit{\sqcup}
		\end{scope}
		\begin{scope}[shift = {(5,7)} ]
			\tiletransmit{\sqcup}
		\end{scope}
		\begin{scope}[shift = {(6,7)} ]
			\blanktile
		\end{scope}
		\begin{scope}[shift = {(7,7)} ]
			\righttile
		\end{scope}
		\begin{scope}[shift = {(0,6)} ]
			\lefttile
		\end{scope}
		\begin{scope}[shift = {(1,6)} ]
			\tilestateincomingright{t}{a}
		\end{scope}
		\begin{scope}[shift = {(2,6)} ]
			\tilestateleft{s}{b}{t}{a}
		\end{scope}
		\begin{scope}[shift = {(3,6)} ]
			\tiletransmit{\sqcup}
		\end{scope}
		\begin{scope}[shift = {(4,6)} ]
			\tiletransmit{\sqcup}
		\end{scope}
		\begin{scope}[shift = {(5,6)} ]
			\blanktile
		\end{scope}
		\begin{scope}[shift = {(6,6)} ]
			\righttile
		\end{scope}
		\begin{scope}[shift = {(0,5)} ]
			\lefttile
		\end{scope}
		\begin{scope}[shift = {(1,5)} ]
			\tiletransmit{a}
		\end{scope}
		\begin{scope}[shift = {(2,5)} ]
			\tilestateup{r}{\sqcup}{s}{b}
		\end{scope}
		\begin{scope}[shift = {(3,5)} ]
			\tiletransmit{\sqcup}
		\end{scope}
		\begin{scope}[shift = {(4,5)} ]
			\blanktile
		\end{scope}
		\begin{scope}[shift = {(5,5)} ]
			\righttile
		\end{scope}
		\begin{scope}[shift = {(0,4)} ]
			\lefttile
		\end{scope}
		\begin{scope}[shift = {(1,4)} ]
			\tilestateright{q_0}{\sqcup}{r}{a}
		\end{scope}
		\begin{scope}[shift = {(2,4)} ]
			\tilestateincomingleft{r}{\sqcup}
		\end{scope}
		\begin{scope}[shift = {(3,4)} ]
			\blanktile
		\end{scope}
		\begin{scope}[shift = {(4,4)} ]
			\righttile
		\end{scope}
		\begin{scope}[shift = {(0,3)} ]
			\lefttile
		\end{scope}
		\begin{scope}[shift = {(1,3)} ]
			\origintile
		\end{scope}
		\begin{scope}[shift = {(2,3)} ]
			\blanktile
		\end{scope}
		\begin{scope}[shift = {(3,3)} ]
			\righttile
		\end{scope}
		\begin{scope}[shift = {(0,2)} ]
			\lefttile
		\end{scope}
		\begin{scope}[shift = {(1,2)} ]
			\midtile
		\end{scope}
		\begin{scope}[shift = {(2,2)} ]
			\righttile
		\end{scope}
		\begin{scope}[shift = {(0,1)} ]
			\seedL
		\end{scope}
		\begin{scope}[shift = {(1,1)} ]
			\seedR
		\end{scope}
		\begin{scope}[shift = {(0,0)} ]
			\seed
		\end{scope}
	\end{tikzpicture}
}
	
	\caption{An example of a tiling of the computation layer. Note that fixing the seed symbol on the bottom position determines completely the rest of the tiling. Here we illustrate a Turing machine with alphabet $\Sigma = \{\sqcup,a,b\}$, $Q = \{q_0,r,s,t\}$ and with $\delta(q_0,\sqcup) =(a,r,1)$, $\delta(r,\sqcup) = (s,b,0)$, $\delta(s,b) = (a,t,-1)$, $\delta(t,a) = (b,s,1)$. }
	\label{fig:machine_example}
\end{figure}

In order to interact with the tentacle layer, we will need to do a few modifications to this layer which will essentially consist in the inclusion and removal of symbols from the alphabet. The first modification has the purpose to enable sending ``commands'' to the tentacle layer. In order to do this we will require our Turing machine $\mathcal{T}$ to have, for every $s \in S$, a special pair of states $q_s,q'_s$ and the transition $\delta(q_s,\sqcup) = (q'_s,\sqcup,0)$. We will use this single step transition to communicate the command ``extend the tentacle on direction $s \in S$''. Similarly, we shall also require a special pair of states $q_{\texttt{D}},q'_{\texttt{D}}$ and the transition $\delta(q_{\texttt{D}},\sqcup) = (q'_{\texttt{D}},\sqcup,0)$. This will be used to send the command ``delete the tentacle''. This will be implemented through the exchange of some tiles in the alphabet. Formally, let $A_{\mathcal{T}}^{\texttt{comm}}$, $A_{\mathcal{T}}^{\texttt{old}}$ be the sets 

\[  \begin{tikzpicture}[scale =1.8]
		\node at (-1.2,0) {$A^{\texttt{comm}}_{\mathcal{T}} =$};
		\begin{scope}[shift = {(0,0)}]
		\draw [thick,decorate, decoration = {brace}] (-0.6,-0.6) --  (-0.6,0.6);
		\node at (0.9,0) {$: s \in S$};
		\draw [thick,decorate, decoration = {brace}] (1.3,0.6) --  (1.3,-0.6);
		\node at (1.5,0) {$\cup$};
		\tilecommandleft{s}
		\end{scope}
		
		\begin{scope}[shift = {(2.3,0)}]
			\draw [thick,decorate, decoration = {brace}] (-0.6,-0.6) --  (-0.6,0.6);
			\node at (0.9,0) {$: s \in S$};
		      \draw [thick,decorate, decoration = {brace}] (1.3,0.6) --  (1.3,-0.6);
            \node at (1.5,0) {$\cup$};
			\tilecommandright{s}
		\end{scope}

        \begin{scope}[shift = {(4.6,0)}]
			\draw [thick,decorate, decoration = {brace}] (-0.6,-0.6) --  (-0.6,0.6);
		      \draw [thick,decorate, decoration = {brace}] (1.8,0.6) --  (1.8,-0.6);
            \node at (0.6,-0.5) {$,$};
            \begin{scope}[shift = {(0,0)}]
                \tilecommandleft{\texttt{D}}
            \end{scope}
            \begin{scope}[shift = {(1.2,0)}]
                \tilecommandright{\texttt{D}}
            \end{scope}
		\end{scope}
  
	\end{tikzpicture} \]
    \[  \begin{tikzpicture}[scale =1.8]
		\node at (-1.2,0) {$A^{\texttt{old}}_{\mathcal{T}} =$};
		\begin{scope}[shift = {(0,0)}]
		\draw [thick,decorate, decoration = {brace}] (-0.6,-0.6) --  (-0.6,0.6);
		\node at (0.9,0) {$: s \in S$};
		\draw [thick,decorate, decoration = {brace}] (1.3,0.6) --  (1.3,-0.6);
        \node at (1.5,0) {$\cup$};
		\tilestateup{q_s}{\sqcup}{q'_s}{\sqcup}
		\end{scope}

        \begin{scope}[shift = {(2.3,0)}]
			\draw [thick,decorate, decoration = {brace}] (-0.6,-0.6) --  (-0.6,0.6);
		      \draw [thick,decorate, decoration = {brace}] (0.6,0.6) --  (0.6,-0.6);
			\tilestateup{q_{\texttt{D}}}{\sqcup}{q'_{\texttt{D}}}{\sqcup}
		\end{scope}

	\end{tikzpicture} \]

We shall add {$A^{\texttt{comm}}_{\mathcal{T}}$ to the alphabet and remove {$A^{\texttt{old}}_{\mathcal{T}}$. This has the additional effect that these kind of transitions will only be able to occur at the origin of the tape. The second modification has the purpose of making the Turing machine able to read information from the tentacle layer. We shall only need to read a single symbol from $A$ in the tentacle layer at a time and react to it. To this end, we shall do a slight modification to our definition of Turing machine. We replace the transition function $\delta$ by a ``conditional transition'' function which also depends on the alphabet $A$, namely, we shall consider now $\delta \colon Q \times \Sigma \times A \to Q \times \Sigma \times \{-1,0,1\}$. All but one kind of transition in our Turing machine will actually ignore the symbol $a \in A$, and thus those transitions will still be represented by those on the alphabet $A^0_{\mathcal{T}}$. More precisely, we will consider two special states $q^0,q^1\in Q$ and for each $a \in A$ we will have a special state $q^a \in Q$ such that \[  \delta(q^a,\sqcup,a) = (q^1,\sqcup,0) \mbox{ and } \delta(q^a,\sqcup,b) = (q^0,\sqcup,0) \mbox{ for } b \neq a.  \]
And those two will be the only ``conditional'' transitions. In order to implement those extra transitions we consider the alphabet $A^{\texttt{read}}_{\mathcal{T}}$ described below. The tiles from $A^{\texttt{read}}_{\mathcal{T}}$ are called \define{conditional transition tiles}.

	\[  \begin{tikzpicture}[scale =1.8]
		\node at (-1.2,0) {$A^{\texttt{read}}_{\mathcal{T}} =$};
		\begin{scope}[shift = {(0,0)}]
		\draw [thick,decorate, decoration = {brace}] (-0.6,-0.6) --  (-0.6,0.6);
		\node at (0.9,0) {$: a \in A$};
		\draw [thick,decorate, decoration = {brace}] (1.3,0.6) --  (1.3,-0.6);
		\node at (1.5,0) {$\cup$};
		\tilestateupconditional{q^a}{\sqcup}{q^1}{\sqcup}{a};
		\end{scope}
		
		\begin{scope}[shift = {(2.3,0)}]
			\draw [thick,decorate, decoration = {brace}] (-0.6,-0.6) --  (-0.6,0.6);
			\node at (1.3,0) {$: a,b \in A, a \neq b$};
			\draw [thick,decorate, decoration = {brace}] (2.1,0.6) --  (2.1,-0.6);
			\tilestateupconditional{q^a}{\sqcup}{q^0}{\sqcup}{b};
		\end{scope}
	\end{tikzpicture} \]
	
	The precise Turing machine $\mathcal{T}$ we will use will be better described after introducing the tentacle layer. For now, we just set the alphabet of the computation layer as $A_{\boldsymbol{C}} = (A^{0}_{\mathcal{T}}\setminus A^{\texttt{old}}_{\mathcal{T}}) \cup A^{\texttt{comm}}_{\mathcal{T}} \cup  A^{\texttt{read}}_{\mathcal{T}} \cup \{ \varnothing\}$ and define the computation layer $\boldsymbol{C}$ as the subset of configurations $\gamma \in (A_{\boldsymbol{C}})^M$ such that, given $(\alpha,\delta,\beta)\in \boldsymbol{A}\times \boldsymbol{D}\times \boldsymbol{B}$ satisfy,
	\begin{enumerate}
		\item \textbf{Tiles overlay with computation zones}: We have that for every $(h,w)\in M$, $\gamma(h,w) = \varnothing$ if and only if $\beta(h,w) = \varnothing$.
		\item \textbf{Seed condition:} For every $h \in H$, then $\gamma(h,\epsilon)$ is the seed tile, that is \[ {\begin{tikzpicture}
			\node at (-2,0) {$\gamma(h,\epsilon) = $};
            \begin{scope}[scale = 2]
               \seed 
            \end{scope}
		\end{tikzpicture}}\] 
		\item \textbf{Wang tile condition}: The symbols in $\gamma$ that are on top of a computation zone in $\boldsymbol{B}$ must match vertically and horizontally as Wang tiles (see~\Cref{fig:machine_example}).
	\end{enumerate}
Note that the Wang tile condition is only enforced on the computation zone, so on every step, as the computation nonuniformly branches to one of the sons, Wang tiles will be matched in that direction, and completely erased in the other.

\subsection{The tentacle layer}

Finally, we shall construct a layer which emulates ``tentacles'' which grow from every position in $H$ and are able to extend along the group and bring back information about the symbols in the alphabet layer back to their base. These tentacles are essentially a ``linked list'' data structure which shares some information.

More precisely, consider the alphabet \[ A_{\boldsymbol{T}} = \left(\{0,1\}\times (S \times S) \times A \times (\{\texttt{D}\}\cup \{ \texttt{G}_s : s \in S\})\right) \cup \{\varnothing\}.  \]

For a symbol $t = (b,(s_1,s_2),a,c) \in A_{\boldsymbol{T}}$ with $t \neq \varnothing$, we denote by
\begin{enumerate}
	\item $\texttt{bit}(t) = b$ the \define{branching bit} of $t$,
	\item $\texttt{prev}(t) = s_1$ the \define{predecessor} of $t$,
	\item $\texttt{next}(t) = s_2$ the \define{successor} of $t$,
	\item $\texttt{symb}(t) = a$ the \define{symbol} of $t$,
	\item $\texttt{comm}(t) = c$ the \define{command} of $t$.
\end{enumerate}

We will refer to three types of tentacle ``parts'' depending on the values of $\texttt{prev}(t)$ and $\texttt{next}(t)$. Recall that we denote by $1_H$ the identity of $H$ and that $1_H \in S$. If $\texttt{prev}(t) = 1_H$, we say that $t$ is a \define{base}. If $\texttt{next}(t) = 1_H$ we say that $t$ is a \define{tip} (some $t$ can be simultaneously bases and tips). Finally, if both $\texttt{prev}(t)$ and $\texttt{next}(t)$ are distinct from $1_H$, we say it is an \define{arm}.

The command $\texttt{D}$ stands for ``delete'' and will force the tentacle to be deleted in the next step. The commands $\texttt{G}_s$ stand for ``grow on direction $s$'' and mean that the tentacle should replace the $\texttt{next}$ value of its tip for $s$ and build a new tip at that position. Note that it is possible to grow on the trivial direction $1_H \in S$, which means that the tentacle will remain unchanged.

We shall first define formally the tentacle layer and explain the meaning of the rules afterwards. If for some $t \in A_{\boldsymbol{T}}$ we invoke any of the five functions defined above, we implicitly mean that it is not $\varnothing$. We shall not write this explicitly to avoid cluttering the definition even more.

Let $(\alpha,\delta,\beta,\gamma)\in \boldsymbol{A}\times \boldsymbol{D}\times \boldsymbol{B}\times \boldsymbol{C}$. We define the \define{tentacle layer} as the set $\boldsymbol{T}$ of configurations $\tau \in (A_{\boldsymbol{T}})^M$ which satisfy the following rules:

\begin{enumerate}
	\item \textbf{Seed rule}: For every $h \in H$, $\texttt{prev}(\tau(h,\epsilon)) = \texttt{next}(\tau(h,\epsilon))=1_H$.
	\item \textbf{Tentacle consistency rules:} For every $(h,w)\in M$, if $\tau(h,w)\neq \varnothing$ then:
	\begin{enumerate}
		\item If $\texttt{next}(\tau(h,w)) = s \neq 1_H$, then $ \texttt{prev}(\tau(h\cdot \texttt{next}(\tau(h,w)),w)) = s^{-1}$.
		\item If $\texttt{prev}(\tau(h,w)) = s \neq 1_H$, then $ \texttt{next}(\tau(h\cdot \texttt{prev}(\tau(h,w)),w)) = s^{-1}$.
		\item $\texttt{bit}(\tau(h,w)) = \texttt{bit}(\tau(h\cdot \texttt{next}(\tau(h,w)),w)) = \texttt{bit}(\tau(h\cdot \texttt{prev}(\tau(h,w)),w))$.
		\item $\texttt{symb}(\tau(h,w)) = \texttt{symb}(\tau(h\cdot \texttt{next}(\tau(h,w)),w)) = \texttt{symb}(\tau(h\cdot \texttt{prev}(\tau(h,w)),w))$.
		\item $\texttt{comm}(\tau(h,w)) = \texttt{comm}(\tau(h\cdot \texttt{next}(\tau(h,w)),w)) = \texttt{comm}(\tau(h\cdot \texttt{prev}(\tau(h,w)),w))$.
	\end{enumerate}
	\item \textbf{Tip reads alphabet}: if $\texttt{next}(\tau(h,w)) = 1_H$ then $\texttt{symb}(\tau(h,w)) = \alpha(h,w)$
	\item \textbf{Base reads branching bit}: if $\texttt{prev}(\tau(h,w)) = 1_H$ and  $\beta(h,w) \in \{ (\texttt{S},b), (\blacktriangleright,b)\}$ then $\texttt{bit}(\tau(h,w)) = b$.
	\item \textbf{Base reads the command}: if $\texttt{prev}(\tau(h,w)) = 1_H$ then $\gamma(h,w)$ is a tile with a squid and:
    \begin{enumerate}
        \item If the symbol at the right of the squid is $\texttt{D}$, then $\texttt{comm}(\tau(h,w)) = \texttt{D}$
        \item If the symbol at the right of the squid is $s\in S$, then $\texttt{comm}(\tau(h,w)) = \texttt{G}_s$
        \item In there is no symbol at the right of the squid, then $\texttt{comm}(\tau(h,w)) = \texttt{G}_{1_H}$.
    \end{enumerate}
    \item \textbf{Conditional transition reads alphabet from base}: if $\texttt{prev}(\tau(h,w)) = 1_H$ and $\gamma(hR_{\delta}(h),w)$ is a conditional transition tile, then the symbol at the middle of the tile must be $\texttt{symb}(\tau(h,w))$.
	\item \textbf{Delete command rule:} if $\texttt{comm}(\tau(h,w)) = \texttt{D}$, let $b = \texttt{bit}(\tau(h,w))$, then:
	\begin{enumerate}
		\item If $\texttt{prev}(\tau(h,w)) \neq 1_H$ (if the symbol is not a base) then $\tau(h,wb) = \varnothing$.
		\item If $\texttt{prev}(\tau(h,w)) \neq 1_H$ (if the symbol is a base) then $\texttt{prev}(\tau(h,wb)) = \texttt{next}(\tau(h,wb)) = 1_H$.
	\end{enumerate}
	\item \textbf{Grow command rule:} if $\texttt{comm}(\tau(h,w)) = \texttt{G}_s$, let $b = \texttt{bit}(\tau(h,w))$, then:
	\begin{enumerate}
		\item If $\texttt{next}(\tau(h,w)) \neq 1_H$ (if the symbol is not a tip) then $\texttt{prev}(\tau(h,wb)) = \texttt{prev}(\tau(h,w))$ and $\texttt{next}(\tau(h,wb)) = \texttt{next}(\tau(h,w))$.
		\item If $\texttt{prev}(\tau(h,w)) \neq 1_H$ (if the symbol is a tip) then $\texttt{prev}(\tau(h,wb)) = \texttt{prev}(\tau(h,w))$,  $\texttt{next}(\tau(h,wb)) = s$. 
  
		Moreover, if $s \neq 1_H$, then $\texttt{prev}(\tau(hs,wb)) = s^{-1}$ and $\texttt{next}(\tau(hs,wb)) = 1_H$. 
	\end{enumerate}
	\item \textbf{No new tentacles rule:} For any $b \in \{0,1$\}, unless it conflicts with the Grow or Delete command rules, then $\tau(h,wb) = \varnothing$.
\end{enumerate}

The seed rule states that at every seed position $(h,\epsilon)$ we initially have a tentacle which is simultaneously a base and a tip. The consistency rules establish that the non-trivial symbols in this alphabet arrange themselves as a collection of (doubly) linked lists. Each element points to the next (through the $\texttt{next}(t)$ function) and to the previous one (through the $\texttt{prev}(t)$ function). Also, each one of these lists, which we shall call ``tentacles'' shares the same information (branching bit, command and symbol).

The next three rules impose where the information of the tentacles comes from. The symbol is read on the tip of the tentacle from the alphabet layer, the branching bit is read on the base of the tentacle from the branching layer (thus the base of a tentacle is always attached to the leftmost position of a computation zone), and the command is read on the base of the tentacle from the computing layer in the tile which has the squid on it. Notice that the base of every tentacle will always be matched up with the squid tiles.

The next rule governs how the computation layer reads from the tentacle the information for its conditional transitions. It states that each time a conditional transition occurs (recall that they always occur at the origin of the computation tape), the conditional symbol in the middle of the tile must be the symbol carried by the base of the tentacle, which is the one read at the tip of the tentacle.

The following three rules determine how the tentacles evolve when increasing the length of $w$. The evolution is done in the tree $\{0,1\}^*$ following the branching bit (which is the same as the one in the computation zone attached to the base). If a tentacle has the command ``delete'' then the whole structure is replaced by $\varnothing$ except for the base, which reverts to its initial state of being both a base and a tip. If a tentacle has the command ``grow in direction $s\in S$'' then its tip changes so that it now points at $s$, and a new tip is created at said position. Note that if $s = 1_H$ the rules impose that the tentacle copies to the branching layer with its directions unchanged. The last rule imposes that besides this prescribed evolution, no new tentacles can occur anywhere else.

The behaviour of a single tentacle is illustrated on~\Cref{fig:tentacle}.

\begin{figure}[ht!]
    \centering
    \begin{tikzpicture}
        
        \begin{scope}[shift = {(0,0)}]
            \draw[thick, black!60, <->] (0,0) to (10,0);
            \draw[thick, black!60,  <->] (2,0.5) to (12,0.5);
            \draw[thick, black!60,  <->] (4,1) to (14,1);
            \draw[thick, black!60,  <->] (0,-0.5) to (8,1.5);
            \draw[thick, black!60,  <->] (2,-0.5) to (10,1.5);
            \draw[thick, black!60,  <->] (4,-0.5) to (12,1.5);
            \draw[thick, black!60,  <->] (6,-0.5) to (14,1.5);
            \draw[fill = white, thick] (2,0) circle (0.2);
            \node at (2,0) {$a$};
            \draw[fill = white, thick] (4,0) circle (0.2);
            \node at (4,0) {$b$};
            \draw[fill = white, thick] (6,0) circle (0.2);
            \node at (6,0) {$c$};
            \draw[fill = white, thick] (8,0) circle (0.2);
            \node at (8,0) {$a$};

            \draw[fill = white, thick] (4,0.5) circle (0.2);
            \node at (4,0.5) {$b$};
            \draw[fill = white, thick] (6,0.5) circle (0.2);
            \node at (6,0.5) {$c$};
            \draw[fill = white, thick] (8,0.5) circle (0.2);
            \node at (8,0.5) {$b$};
            \draw[fill = white, thick] (10,0.5) circle (0.2);
            \node at (10,0.5) {$b$};

            \draw[fill = white, thick] (6,1) circle (0.2);
            \node at (6,1) {$c$};
            \draw[fill = white, thick] (8,1) circle (0.2);
            \node at (8,1) {$c$};
            \draw[fill = white, thick] (10,1) circle (0.2);
            \node at (10,1) {$a$};
            \draw[fill = white, thick] (12,1) circle (0.2);
            \node at (12,1) {$c$};
            
        \end{scope}

        \draw[thick, shorten <=0.2cm, black!20, ->] (2,0) to (2,7); 
        \draw[thick, shorten <=0.2cm, black!20, ->] (4,0) to (4,7); 
        \draw[thick, shorten <=0.2cm, black!20, ->] (6,0.5) to (6,7.5); 
        \draw[thick, shorten <=0.2cm, black!20, ->] (8,0.5) to (8,7.5); 
        \node at (0.5,1) {$\texttt{G}_{(1,0)}$};
        \node at (0.5,2) {$\texttt{G}_{(0,1)}$};
        \node at (0.5,3) {$\texttt{G}_{(0,0)}$};
        \node at (0.5,4) {$\texttt{G}_{(1,0)}$};
        \node at (0.5,5) {$\texttt{D}$};
        \draw[fill = white, thick] (2,1) circle (0.2);
            \node at (2,1) {$a$};
        \draw[thick, dotted, ->, black!80] (1.8,1) arc (20:340:2mm);
        \draw[thick,->, black] (2.2,1) arc (160:-160:2mm);
        \draw[fill = white, thick] (2,2) circle (0.2);
            \node at (2,2) {$b$};
        \draw[thick, ->,shorten >=0.2cm,shorten <=0.2cm] (2,2) to [out=15,in=165] (4,2);
        \draw[thick, dotted,  ->,black!80, shorten >=0.2cm,shorten <=0.2cm] (4,2) to [out=195,in=-15] (2,2);
        \draw[fill = white, thick] (4,2) circle (0.2);
            \node at (4,2) {$b$};
        \draw[thick, dotted, ->, black!80] (1.8,2) arc (20:340:2mm);
        \draw[thick,->, black] (4.2,2) arc (160:-160:2mm);
        \begin{scope}[shift = {(0,2)}]
           \draw[fill = white, thick] (2,1) circle (0.2);
           \draw[fill = white, thick] (4,1) circle (0.2);
            \draw[fill = white, thick] (6,2) circle (0.2);
            \node at (4,1) {$c$};
            \node at (6,2) {$c$};
            \node at (2,1) {$c$};
            \draw[thick, ->,shorten >=0.2cm,shorten <=0.2cm] (2,1) to [out=15,in=165] (4,1);
            \draw[thick,  dotted, ->,black!80, shorten >=0.2cm,shorten <=0.2cm] (4,1) to [out=195,in=-15] (2,1);
            \draw[thick, dotted, ->, black!80] (1.8,1) arc (20:340:2mm);
            \draw[thick, ->,shorten >=0.2cm,shorten <=0.2cm] (4,1) to [out=40,in=190] (6,2);
            \draw[thick, dotted,  ->,black!80, shorten >=0.2cm, shorten <=0.2cm] (6,2) to [out=220,in=10] (4,1);
            \draw[thick,->, black] (6.2,2) arc (160:-160:2mm); 
        \end{scope}
        \begin{scope}[shift = {(0,3)}]
           \draw[fill = white, thick] (2,1) circle (0.2);
           \draw[fill = white, thick] (4,1) circle (0.2);
            \draw[fill = white, thick] (6,2) circle (0.2);
            \node at (4,1) {$c$};
            \node at (6,2) {$c$};
            \node at (2,1) {$c$};
            \draw[thick, ->,shorten >=0.2cm,shorten <=0.2cm] (2,1) to [out=15,in=165] (4,1);
            \draw[thick,  dotted, ->,black!80, shorten >=0.2cm,shorten <=0.2cm] (4,1) to [out=195,in=-15] (2,1);
            \draw[thick, dotted, ->, black!80] (1.8,1) arc (20:340:2mm);
            \draw[thick, ->,shorten >=0.2cm,shorten <=0.2cm] (4,1) to [out=40,in=190] (6,2);
            \draw[thick, dotted,  ->,black!80, shorten >=0.2cm, shorten <=0.2cm] (6,2) to [out=220,in=10] (4,1);
            \draw[thick,->, black] (6.2,2) arc (160:-160:2mm); 
        \end{scope}
        \begin{scope}[shift = {(0,4)}]
           \draw[fill = white, thick] (2,1) circle (0.2);
           \draw[fill = white, thick] (4,1) circle (0.2);
            \draw[fill = white, thick] (6,2) circle (0.2);
            \draw[fill = white, thick] (8,2) circle (0.2);
            \node at (4,1) {$b$};
            \node at (6,2) {$b$};
            \node at (2,1) {$b$};
            \node at (8,2) {$b$};
            \draw[thick, ->,shorten >=0.2cm,shorten <=0.2cm] (2,1) to [out=15,in=165] (4,1);
            \draw[thick,  dotted, ->,black!80, shorten >=0.2cm,shorten <=0.2cm] (4,1) to [out=195,in=-15] (2,1);
            \draw[thick, dotted, ->, black!80] (1.8,1) arc (20:340:2mm);
            \draw[thick, ->,shorten >=0.2cm,shorten <=0.2cm] (4,1) to [out=40,in=190] (6,2);
            \draw[thick, dotted,  ->,black!80, shorten >=0.2cm, shorten <=0.2cm] (6,2) to [out=220,in=10] (4,1);
            \draw[thick, ->,shorten >=0.2cm,shorten <=0.2cm] (6,2) to [out=15,in=165] (8,2);
            \draw[thick,  dotted, ->,black!80, shorten >=0.2cm,shorten <=0.2cm] (8,2) to [out=195,in=-15] (6,2);
            \draw[thick,->, black] (8.2,2) arc (160:-160:2mm); 
        \end{scope}

        \begin{scope}[shift = {(0,5)}]
        \draw[fill = white, thick] (2,1) circle (0.2);
            \node at (2,1) {$a$};
        \draw[thick,dotted, ->, black!80] (1.8,1) arc (20:340:2mm);
        \draw[thick,->, black] (2.2,1) arc (160:-160:2mm);
        \end{scope}

    \end{tikzpicture}
    \caption{An illustration of the dynamics of a tentacle. In the base we have $H = \ZZ^2$ filled with symbols from $A = \{a,b,c\}$. The tentacle starts at the bottom left position and evolves according to the command written on the left. The $\texttt{next}$ arrows are bold, while the $\texttt{prev}$ arrows are dotted. At every step, the symbol shown on the tentacle is the one at the tip. In this figure the tentacle reaches the position $(2,1)$ from its starting position and then gets deleted.}
    \label{fig:tentacle}
\end{figure}
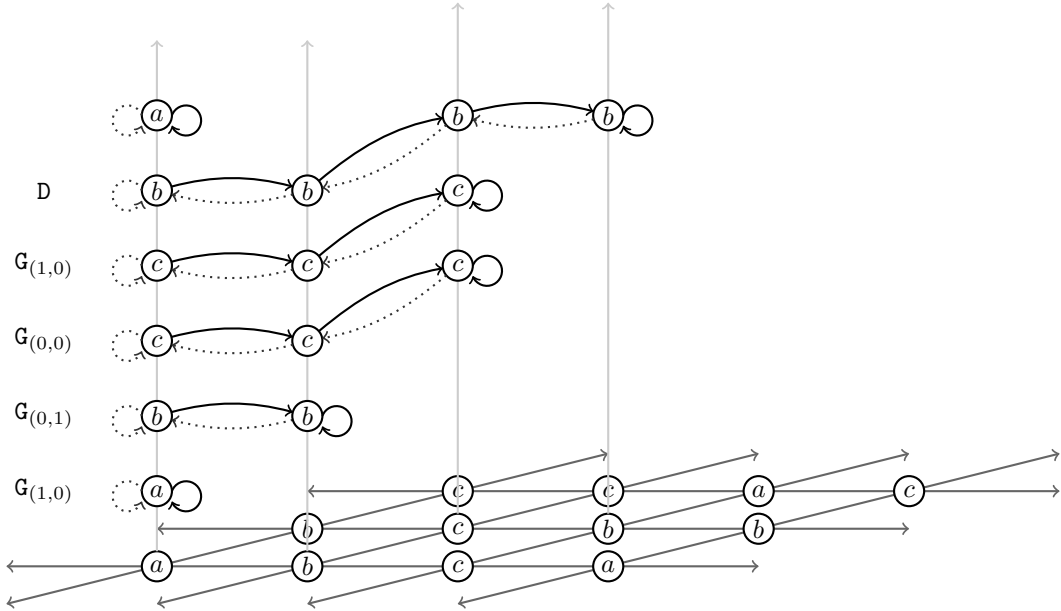

\subsection{Description of the Turing machine} Recall that the word problem of $H$ is decidable, that is, there exists an algorithm which on input $w \in S^*$ accepts if and only if $w$ represents the identity in $H$. From this algorithm, one can easily construct another which takes $w \in S^*$ and outputs $u \in S^*$ which represents the same element as $w$ and has minimal length (for instance, running the algorithm for the word problem on $wu^{-1}$ for every $u$ of length at most $|w|$ and choosing the smallest for which it accepts). Such a word $u$ will be referred to as a ``geodesic''.

Also as $X$ is an effectively closed (by patterns) subshift, there exists an algorithm which on input $n \in \NN$, outputs a pattern coding $c_n$ such that the collection $\mathcal{C} = \{c_n : n \in \NN\}$ defines $X$. Recall that a pattern coding can be thought of as a finite list $c_n = \{ (w_i,a_i)\}_{i \in I}$ with $w_i \in S^*$ and $a_i \in A$. Running this algorithm along with the one which replaces a word by a geodesic, we may assume without loss of generality that the algorithm is chosen such that all $w_i$ occurring in each $c_n$ are geodesics.

Next we will describe the Turing machine $\mathcal{T}$ used to define the computation layer. We will build $\mathcal{T}$ so that it works on a one-sided tape and such that it never writes over the origin (it may move and change states at the origin). Recall also that we ask $\mathcal{T}$ to have special states $q_s,q_s'$ for each $s \in S$, $q_{\texttt{D}},q'_{\texttt{D}}$, $q^a$ for each $a \in A$ and $q^0,q^1$; and that there are conditional transitions from each $q_a$ to either $q^0$ or $q^1$ depending on a symbol from $A$ (which will be read from the tentacle layer) as follows \[ \delta(q^a,\sqcup,b) = \begin{cases}
    (q^0,\sqcup, 0) & \mbox{ if } b \neq a\\
    (q^1,\sqcup, 0) & \mbox{ if } b = a\\
\end{cases}  \] 

We define $\mathcal{T}$ as a Turing machine which executes the following algorithm. 

\begin{enumerate}
	\item Initialize $n = 0$ and execute the following loop:
	\begin{enumerate}
		\item Compute $c_n = \{(w_i,a_i)\}_{i \in I}$ (recall that each $w_i$ is a geodesic).
        \item Initialize $\texttt{DANGER} = 1$
        \item For each $i \in I$, let $m = |w_i|$ and write $w_i = s_1s_2\dots s_m \in S^*$.
        \begin{enumerate}
            \item Wait for at least $\left( 2\frac{|S|^{2m+1}-1}{|S|-1}\right) + \left\lceil\log_2\left( \frac{|S|^{2m+1}-1}{|S|-1}\right)\right\rceil + 1$ steps.
            \item Initialize $j = 1$ and execute the following loop:
            \begin{enumerate}
                \item If $j = m+1$, break the loop (go to (iii)).
                \item Move the head to the origin and go to state $q_{s_j}$ (this sends the instruction ``grow by $s_j$'' to the tentacle layer).
                \item replace $j$ by $j+1$ and restart the loop (go to (A)).
            \end{enumerate}
            \item Move the head to the origin, go to state $q_{a_i}$ and execute the conditional transition (this reads the symbol from the tentacle layer, goes to $q^1$ if it is $a_i$ and to $q^0$ otherwise).
            \item If the current state is $q^0$, replace the value of $\texttt{DANGER}$ by $0$.
            \item Move the head to the origin and go to state $q_{\texttt{D}}$ (this sends the instruction ``delete'' to the tentacle layer).
        \end{enumerate}
        \item If $\texttt{DANGER}=1$, go to the final state $q_F$ (a forbidden pattern has been detected).
        \item Replace $n$ by $n + 1$ and restart the loop (go to (a)).
	\end{enumerate}
\end{enumerate}

It is clear that such an algorithm can be implemented by a Turing machine (with conditional transitions). The behavior is the following, if the leftmost position of the computation zone is attached to $h \in H$, it starts computing in order the forbidden pattern codings $c_n$. As soon as it finds one, it interacts with the tentacle layer to determine whether $h^{-1}(\alpha|_{H \times \{\epsilon\}}) \in [c_n]$. If it is the case, then at the end of the loop it will have $\texttt{DANGER}=1$ and the machine will go to the final state $q_F$. Otherwise, it checks the next pattern coding in the list.

The only instruction which we need to explain is the ``wait'' one. This simply means that the Turing machine will initialize a counter and count up to that number (thus delaying the rest of the algorithm for at least that number of steps). The purpose of this seemingly useless instruction is to allow enough time for branchings to occur in order to allow the existence of a configuration where the tentacles attached to different computation zones do not collide.

\subsection{Proof of the main theorem}

Now we are ready to prove the main result. We define $Z$ as the subset of all configurations $z = (\alpha,\delta,\beta,\gamma,\tau)\in \boldsymbol{A}\times \boldsymbol{D}\times \boldsymbol{B}\times \boldsymbol{C}\times \boldsymbol{T}$ such that no tile occurring in $\gamma$ carries the final state $q_F$. From the definition of each layer it is direct that $Z$ is a rooted SFT.

\begin{lemma}\label{lem:Z_subfactors_onto_X}
    The rooted SFT $Z$ root-factors onto $X$.
\end{lemma}

\begin{proof}
    Let $\phi \colon Z\ \to A^{H}$ be the map such that if $z = (\alpha,\delta,\beta,\gamma,\tau) \in Z$, then $\phi(z) = \alpha(h,\epsilon)$. Clearly $\phi$ is given by an alphabet map $\Phi$ (the projection to the first component) as in~\Cref{def:subfactors}, therefore it suffices to show that $\phi(Z)=X$.

    Let $x \in X$. We shall construct $z \in Z$ such that $\phi(z)=x$. Let $\alpha \in A^M$ be such that $\alpha(h,w) = x(h)$ for every $h \in H$. By our choice of generating set $S$, there exists $\delta \in \boldsymbol{D}$ such that $T_{\delta}$ is a translation-like action of $\ZZ$ on $H$. Clearly $(\alpha,\delta) \in \boldsymbol{A}\times \boldsymbol{D}$.

    Notice that a configuration $\beta \in \boldsymbol{B}$ is entirely determined by $\delta$ and the choices of branching bits. We are going to produce an assignment of branching bits which will ensure that no computation zones overlap, and that no tentacles will collide. To this end, let $f\colon \NN \to \NN$ be given by $f(n)=2^n -1$, and consider a bijective map $g\colon \NN\setminus f(\NN) \to (H\setminus \{1_H\})\times \{-1,+1\}$ with the property that elements of $H$ occur in the image of $g$ with non-decreasing word-length with respect to $S$.

    We define $\beta \in \boldsymbol{B}$ as the configuration produced by the following choice of branching bits on each $w \in \{0,1\}^*$:
    \begin{enumerate}
        \item If $|w| =2^n-1$ for some $n \in \NN$, assign alternating branching bits to the symbols $\texttt{S}$ or $\blacktriangleright$ in each copy of $\ZZ$ induced by $\delta$.
        \item Otherwise, let $g(|w|) = (h,m)$. Let $(t_i)_{i \in I}$ be a set of left coset representatives of the cyclic subgroup of $H$ generated by $h$.
        \begin{itemize}
            \item If $h$ is not a torsion element, assign the branching bit $0$ to elements of the form $(t_ih^{2k},w)$ for $k \in \ZZ$ which are paired with $\blacktriangleright$, and the branching bit $1$ to elements of the form $(t_ih^{2k+1},w)$ for $k \in \ZZ$ which are paired with $\blacktriangleright$.
            \item If $h$ is a torsion element, let $\kappa$ be the smallest positive integer such that $h^{\kappa}=1_{H}$. Assign the branching bit $0$ to elements of the form $(t_i h^{m \cdot k},w)$ for even values of $k \in \{0,\dots,\kappa-1\}$ which are paired with $\blacktriangleright$. Assign the branching bit $1$ to elements of the form $(t_i h^{mk},w)$ for odd values of $k \in \{0,\dots,\kappa-1\}$ which are paired with $\blacktriangleright$.
        \end{itemize}
    \end{enumerate}

    The assignment above for $|w| =2^n-1$ ensures that no computation zones collide and thus that $\beta$ is well defined. The second choice is slightly more subtle. At $g(|w|)= (h,m)$ we want to force elements of $H$ which are separated by $h$ to lie in different branches from that point onwards. This is easy enough to do if $h$ is a non-torsion element, as we can just assign alternating branching bits to every leftmost corner of a computation zone in the subgroup induced by $H$. However, if $h$ is a torsion element of odd order this cannot be done. To fix this, we do the alternating assignment twice, once going forward, and once going backwards. This ensures that after both $(h,-1)$ and $(h,+1)$ pass, then no elements separated by $h$ lie in the same branch.

    We have thus $(\alpha,\delta,\beta)\in \boldsymbol{A}\times \boldsymbol{D}\times \boldsymbol{B}$. These three configurations already determine $(\gamma,\tau)$ entirely if they exist, and the only way that they might not exist is if two tentacles collide. We claim that the choice of branchings given above ensures that no tentacles collide. Indeed, notice that if $B_m$ is the ball of radius $m$ in $H$ with respect to the word metric induced by $S$, then \[|B_m| \leq \sum_{k = 0}^m |S|^k = \frac{|S|^{m+1}-1}{|S|-1}.\]

    As we chose $g$ such that the images are non-decreasing in word length, it follows that in order to make sure that every $h \in B_{2m}$ has already occurred twice in the image of $g$, it suffices to be in an element of $\{0,1\}^*$ with length at least \[     \frac{2(|S|^{2m+1}-1)}{|S|-1} + \left\lceil \log_2\left(\frac{|S|^{2m+1}-1}{|S|-1} \right)\right\rceil +1 \geq 2|B_{2m}| +\left\lceil \log_2\left(2|B_{2m}|\right)\right\rceil.. \]
    The log term is needed due to the alternating branching bits assigned on steps $2^n-1$ that keep tapes from colliding on a single $T$-orbit on the directions layer.
    
    This is is precisely the number of steps we asked the algorithm to wait in the description of our Turing machine. Therefore when the machine sends information to the tentacle layer to grow from $h$ to $h{\underline{w}}$, we already know that every position in every ball of size $2m$ in the group $H$ is already on a distinct branch, thus the tentacles will not collide with each other. Moreover, as the word produced in each pattern coding $c_n$ are geodesics, the tentacles with not collide with themselves either.

    From the argument above, it follows that we have $z = (\alpha,\delta,\beta,\gamma,\tau)\in \boldsymbol{A}\times \boldsymbol{D}\times \boldsymbol{B}\times \boldsymbol{C}\times \boldsymbol{T}$ and that $\phi(z)=x$. We just need to argue that the final state is never reached in $\gamma$ in order to have $z \in Z$. Indeed, as $x \in X$, by definition for every forbidden pattern coding $c_n = (w_i,a_i)_{i \in I}$ and every $h \in H$ we have that there is $i \in I$ such that $x(h\underline{w_i})\neq a_i$. This means that the algorithm when run on the loop corresponding to $(w_i,a_i)$ will have the tentacle tip on $h\underline{w_i}$ looking at the symbol $x(h\underline{w_i})$. Therefore the next step of the algorithm will change the value of the $\texttt{DANGER}$ variable to $0$ and not go to the final state $q_F$. As this holds true for every $h \in H$ and $n \in \NN$, it follows that $q_F$ is not reached in any computation zone and thus $z \in Z$.

    Conversely, let $z = (\alpha,\delta,\beta,\gamma,\tau) \in Z$ and let $x = \phi(z)$. Suppose that $x \notin X$ it follows that there exists $h \in H$ and $n \in \NN$ such that if $c_n = (w_i,a_i)_{i \in I}$ then for every $i \in I$ we have $x(h\underline{w_i}) = a_i$. Consider the computation zone whose leftmost position is on $h$. Eventually the algorithm produces the pattern coding $c_n$ and each iteration of the loop keeps the variable $\texttt{DANGER}=1$ unchanged. Therefore at the end of the loop the machine goes to state $q_F$, which raises a contradiction because we assumed that $z \in Z$. Thus $x \in X$.\end{proof}

\begin{proof}[Proof of~\Cref{thm:main_theorem}]
    Let $H$ be a finitely generated group, $N$ a non-amenable group and $X$ an effectively closed $H$-subshift. If $H$ is finite, then $X$ is an SFT, and thus so is its free extension to $H \times N$. If $H$ is infinite and with decidable word problem,~\Cref{lem:Z_subfactors_onto_X} ensures that there exists a rooted SFT $Z$ which root-factors onto $X$. By~\Cref{lem:reduction} we obtain that the free extension of $X$ to $H \times N$  is a sofic subshift.
\end{proof}
    
\section{Applications}\label{sec:aplications}}

The first consequence of~\Cref{thm:main_theorem} is rather straightforward. If $N$ is finitely generated, we can grab a free extension and turn it into a trivial extension by forcing that symbols are constant along the generators of $N$. 

\begin{corollary}\label{cor:simulation_trivial_expansive}
    Let $H,N$ be finitely generated groups. Let $H$ have decidable word problem and let $N$ be non-amenable. Then the trivial extension of every effectively closed $H$-subshift to $G=H\times N$ is sofic.
\end{corollary}

\begin{proof}
    Let $X$ be an effectively closed $H$-subshift. By~\Cref{thm:main_theorem}, the free extension $\widetilde{X}$ of $X$ to $G$ is sofic. Let $S\Subset N$ be a finite set of generators and consider the $G$-subshift $Y$ given by \[  Y = \{ x \in \widetilde{X} : \mbox{ for every } s \in S \mbox{ and } (h,n)\in G, x(h,n) = x(h,ns)\}.\]
    The extra rule is local and thus $Y$ is still a sofic $G$-subshift. It is clear that $Y$ is precisely the trivial extension of $X$ to $G$.\end{proof}

\subsection{Simulation of actions of non-amenable groups}\label{sec:simulationproduct}


The purpose of this section is to prove a simulation theorem for effectively closed actions of a finitely generated non-amenable group $N$. More precisely, we shall prove~\Cref{thm:simulation_nonamenable}. The main idea is to construct an effectively closed $H$-subshift such that every configuration encodes a single element of the set representation of an effectively closed action $N\curvearrowright X$ in such a way that this coding is invariant under shifts by elements of $H$. By~\Cref{thm:main_theorem} the free extension of said subshift to $H \times N$ is sofic, thus we just need to add some local rules to force each $H$-coset to communicate with its neighbors in such a way that the natural topological factor of the $N$-subaction is precisely $N\curvearrowright X$.

\begin{proof}[Proof of~\Cref{thm:simulation_nonamenable}]
Let $X \subseteq A^\NN$ and $N \curvearrowright X$ be an effectively closed action. Let $S$ be a finite generating set of $N$ which contains the identity and for which its set representation $\operatorname{Rep}(G \curvearrowright X, S)$ is an effectively closed set (see~\Cref{def:ECaction}). Following the idea from~\cite{BS2018}, we embed this set in an effectively closed $\ZZ$-subshift using Toeplitz codings. Let $B$ be an alphabet. Consider the map $\psi\colon B^\NN\to (B \cup \{ \$ \} )^\ZZ$ defined by
\[
\mathbf{\psi}(y)_j=\begin{cases}
y_n  & \textrm{ if $j=3^n$ mod $3^{n+1}$ for some $n \in \NN$} \\
\$ & \textrm{ otherwise;}
\end{cases}
\]

That is, the image of some $y \in B^{\NN}$ looks as follows (we show the restriction to $\{0,\dots, 30\}$) \[\psi(y)=\dots \$y_0\$y_1y_0\$\$y_0\$y_2y_0\$y_1y_0\$\$y_0\$\$y_0\$y_1y_0\$\$y_0\$y_3y_0\$ y_1 \dots\]

For a set $Y\subset B^{\NN}$, consider the $\ZZ$-subshift $\mathbf{Top}(Y)$ obtained as the topological closure of the union of the orbits of $\psi(y)$ for some $y \in Y$ under the shift action. One can show (see Section 3 of~\cite{BS2018}) that if $Y$ is an effectively closed set, then $\mathbf{Top}(Y)$ is an effectively closed $\ZZ$-subshift which admits a continuous (and computable) map $\phi \colon \mathbf{Top}(Y) \to Y$ which is constant on orbits. Indeed, in every configuration $x\in\mathbf{Top}(B^{\NN})$ which is in the orbit closure of some $\psi(y_0y_1y_2\dots)$, notice that the symbol $\$$ can be preceded only either by $\$$ or by $y_0\in B$ which appears with period three. If we remove the subwords $y_0\$$ in $x$, one obtains a  new element of $\mathbf{Top}(B^{\NN})$ which now comes from the orbit closure of $\psi(y_1y_2y_3\dots)$. The map  $\phi \colon \mathbf{Top}(Y) \to Y$ is obtained by repeating the procedure recursively. 

Now we can set $B = A^S$ and $Y = \operatorname{Rep}(G \curvearrowright X, S)$ and thus $T = \mathbf{Top}(Y)$ is an effectively closed $\ZZ$-subshift which admits a continuous map $\phi \colon T\to \operatorname{Rep}(G \curvearrowright X, S)$.

By the result of Seward we used in the previous section (\cite[Theorem 1.4]{Seward2014}), there exists a translation-like action of $\ZZ$ on $H$. Fix a finite symmetric set $K$ of generators of $H$ such that there exists a translation-like action $f\colon H \to H$ with $h^{-1}f(h) \in K$. Consider the alphabet $A_{D}=K \times K$. For $a = (k,k')\in K \times K$ denote $\ell(a) = k$, $r(a)=k'$ where $\ell$ and $r$ stand for ``left'' and ``right''.

Consider the $H$-SFT $D$ as the space of all configurations $d \in (A_{D})^{H}$ such that for any $h \in H$, \begin{enumerate}
	\item if $s = r(d(h))$ and $u = \ell(d(hs))$ then $s = u^{-1}$;
	\item if $s = \ell(d(h))$ and $u = r(d(hs))$ then $s = u^{-1}$.
\end{enumerate}

Both rules are local and thus $D$ is an $H$-SFT which encodes actions of $\ZZ$ on $H$ which are bounded by $K$. Given $d \in D$ and $h \in H$, denote by $R_d(h) = h \cdot r(d(h))$ and $L_d(h) = h \cdot \ell(d(h))$, and notice that $L_d$ is the inverse of $R_d$ as bijections on $H$.

Finally, Let $D[T] \subseteq D\times(A^{S}\cup\{\$\})^H$ be the $H$-subshift such that $(d,y)\in D[T]$ if and only if for every $h,h'\in H$ one has \[ \left( y(R_d^n(h))  \right)_{n \in \ZZ} \in T \mbox{ and }  \phi(\left( y(R_d^n(h))  \right)_{n \in \ZZ}) = \phi(\left( y(R_d^n(h')) \right)_{n \in \ZZ}).   \]

In simpler words, in each copy of $\ZZ$ induced by $d$ there is an element of $T$, and all of these elements of $T$ code the same element of $\operatorname{Rep}(G \curvearrowright X, S)$. As $D$ is an SFT, and $T$ is effectively closed, it follows that $D[T]$ is an effectively closed $H$-subshift. Moreover, as $D$ contains one element which encodes the translation-like action $f$, it follows that $D[T]$ is nonempty.

Let $\widetilde{D}[T]$ be the free extension to $H \times N$ of the effective $H$-subshift $D[T]$. As the word problem of $H$ is decidable and $N$ is non-amenable, it follows by Theorem~\ref{thm:main_theorem} that $\widetilde{D}[T]$ is a sofic subshift. Let us introduce some notation, we shall write configurations of $\widetilde{D}[T]$ as pairs $(d,t) \in (A_D)^{H \times N} \times (A^S \cup \{\$\})^{H \times N}$. Also, for $b \in A^S \cup \{\$\}$ and $s \in S$ we write $\widetilde{\pi}_s(b)$ for its $s$-th coordinate if $b \in A^S$ and set $\widetilde{\pi}_s(\$)=\$$.

Finally, consider $H \times N$-subshift $W\subset \widetilde{D}[T]$ which consists on the configurations $(d,y) \in \widetilde{D}[T]$ that satisfy the following two additional local rules for every $(h,n) \in H \times N$:

\begin{enumerate}
    \item For every $s \in S$, $d(h,n) = d(h,ns)$.
    \item For every $s \in S$, $\widetilde{\pi}_s(y(h,n)) = \widetilde{\pi}_{1_N}(y(h,ns^{-1}))$.
\end{enumerate}

In other words, we force the restriction of $d$ in each $H$-coset to be exactly the same, and the force the $s$-th coordinates in every position to correspond to the $1_N$-th coordinates in the positions reached by moving by $s\in S$. Notice that as $d$ is the same in every coset, the map $R_d$ is well defined as above. These two rules are clearly local and thus $W$ is sofic as well. Let $Z$ be a subshift of finite type on $H\times N$ which factors onto $W$ through some map $\varphi\colon Z \to W$.

Let $\gamma \colon W \to T$ be the map given by $\gamma(d,y) = \left(y(R_d^n(1_H),1_N) \right)_{n \in \NN}$. That is, the map which assigns to $(d,y) \in W$ the element from the set representation obtained by following the identity.

Finally, consider $\rho \colon Z \to X$ defined by $\rho=\pi_{1_N}\circ\phi \circ \gamma \circ \varphi$. As all maps are continuous and surjective, it follows that $\rho$ is continuous and surjective. We just need to show that it factors onto the trivial extension of $N \curvearrowright X$.

Let $z \in Z$, $h \in H$ and $\varphi(z) = (d,y) \in W$. By the second rule of the construction of $D[T]$ we have that $\phi\circ \gamma((h,1_N)\cdot (d,y)) = \phi\circ \gamma (d,y)$ and thus $\rho((h,1_N)\cdot z) = \rho(z)$.

On the other hand, let $s \in S$ and notice that one has for any $k \in \NN$
\[\gamma(d,y)(k) = y(R_d^n(1_H),1_N) \mbox{ and } \gamma((1_H,s)\cdot(d,y))(k) = y(R_d^n(1_H),s^{-1}). \]
By the second local rule one obtains that $\widetilde{\pi}_{1_N}(\gamma((1_H,s)\cdot(d,x))(n)=\widetilde{\pi}_s(\gamma(d,x)(n)$ from where we obtain that
\[\pi_{1_N}\circ\phi\circ \gamma ((1_H,s)\cdot(d,y))= \pi_{s}\circ\phi\circ \gamma (d,y).  \]

From the above identity and the fact that $\pi_s = s \cdot \pi_{1_N}$ on the set representation, one concludes that $\rho( (1_H,s)\cdot z) = s \cdot \rho(z)$.

Putting together the two identities and the fact that $S$ generates $N$, we obtain that for every $(h,n)\in H \times N$, $\rho((h,n)\cdot z) = n \cdot \rho(z)$ and thus $Z$ factors onto the trivial extension of $N \curvearrowright X$ to $H \times N$.\end{proof}

\subsection{Strongly aperiodic SFT for direct products with a non-amenable component}

\begin{definition}
A $G$-subshift is \define{strongly aperiodic} if $G$ acts freely on it, that is, if $gx = x$ for some $x \in X$ then $g = 1_{G}$.
\end{definition}

In~\cite{ABT2018} it was shown that every finitely generated group with decidable word problem admits a nonempty effectively closed subshift which is strongly aperiodic. We shall use this result to produce nonempty strongly aperiodic SFTs on groups of the form $G = H \times N$ where both groups are infinite, finitely generated, have decidable word problem and $N$ is non-amenable.

\begin{proof}[Proof of~\Cref{thm:aperiodic}]
 As $H,N$ have decidable word problem, by~\cite[Theorem 2.6]{ABT2018}, the groups $H$ and $N$ admit nonempty strongly aperiodic effective subshifts $X_1$ and $X_2$ respectively. By~\Cref{thm:main_theorem} there exists an $(H\times N)$-SFT $Z_1$ which factors onto the free extension of $H \curvearrowright X_1$, and by~\Cref{thm:simulation_nonamenable}, there exists there exists an $(H\times N)$-SFT $Z_2$ which factors onto the trivial extension of $N \curvearrowright X_1$. Let $Z = Z_1\times Z_2$ with the product shift action.
 
Let $(h,n)\in H \times N$ and $(z_1,z_2)\in Z$ such that $(h,n)\cdot (z_1,z_2)=(z_1,z_2)$. One has $(h,n)\cdot z_2=(1_H,n)\cdot z_2=z_2$. As $X_2$ is strongly aperiodic, applying the topological factor map we obtain that $n=1_N$. Then one has $(h,1_N)\cdot z_1 = z_1$. As $X_1$ is strongly aperiodic, applying the respective topological factor map we obtain that $h=1_H$. Thus $Z$ must be a strongly aperiodic SFT.\end{proof}

In the particular case of $H=\ZZ$ and $N = F_2$ is the free group on two generators, we partially recover a result of~\cite{AuBiHTB_2022} which shows the existence of a nonempty, minimal and strongly aperiodic SFT on $\ZZ \times F_2$ (we don't get minimality with our technique). In the case where $H$ is also non-amenable we recover a particular case of~\cite[Corollary 8.18]{BaSaSa_2021}.

\section{Questions}\label{sec:questions}

In~\Cref{thm:main_theorem} we used the hypothesis that the group $H$ has decidable word problem. This was fundamentally used to ensure that every word occurring in a pattern coding is geodesic and thus that tentacles will not collide with themselves. We do not know whether this condition can be removed or at least replaced with the milder condition that $H$ is recursively presented (or co-recursively presented).

\begin{question}
    Does~\Cref{thm:main_theorem} hold if $H$ does not have decidable word problem?
\end{question}

In~\Cref{cor:simulation_trivial_expansive} we showed that under the additional condition that $N$ is finitely generated, the trivial extension of an effectively closed subshift on $H$ to $H \times N$ is sofic. We were not able to extend this result to the trivial extension of a (non-expansive) effectively closed action of $H$ using our main result.

\begin{question}
    Let $H,N$ be infinite and finitely generated groups with decidable word problem and such that $N$ is non-amenable. Let $H \curvearrowright X$ be an effectively closed action. Is the trivial extension of $H \curvearrowright X$ to $H \times N$ the topological factor of a subshift of finite type?
\end{question}

Another questions concerns completing the picture for free extensions of subshifts. In~\Cref{tab:my_label} we summarize what we know for the free extension of an $H$-subshift to $H \times K$. In the table we suppose both $H,K$ are infinite, finitely generated and with decidable word problem.

\begin{table}[]
    \centering
    \begin{tabular}{|c|c|c|}
    \hline
    \diagbox{$H$}{$K$} & amenable & non-amenable \\
    \hline
        amenable & not always sofic & always sofic \\
        \hline
        non-amenable & depends & always sofic \\
        \hline
    \end{tabular}
    \vspace{4ex}
    \caption{Soficity of free extensions of effectively closed $H$-subshifts to $H\times K$}
    \label{tab:my_label}
\end{table}

In the case where both $H$ and $K$ are non-amenable, we actually have that every effectively closed $H\times K$-subshift is sofic, and thus the result also follows from~\cite{BaSaSa_2021}. In the case where both groups are amenable we constructed a single counterexample, but indeed we have no counterexamples to the generalization of Jeandel's question of whether soficity of the free extension implies soficity of the subshift on $H$.

\begin{question}
    Let $H$ and $K$ be two infinite, finitely generated and amenable groups with decidable word problem. Let $X$ be an effectively closed $H$-subshift whose free extension to $H \times K$ is sofic. Is $X$ sofic?
\end{question}

The remaining case is when $H$ is non-amenable and $K$ is amenable. Here we can have both behaviors: if $H$ is a self-simulable group (see~\cite{BaSaSa_2021}, for instance $H= F_2 \times F_2$), then every effectively closed subshift on $H$ is sofic, and thus so are their free extensions to $H \times K$. In the other hand, if $H = F_2$ and $K = \ZZ$, then the construction in~\cite[Proposition 3.4]{BaSaSa_2021} provides an example of an effectively closed $F_2$-subshift (again a variant of the mirror shift) whose free extension to $F_2 \times \ZZ$ is not sofic.

\begin{question}
    Let $H$ be an infinite amenable group. For which non-amenable groups $N$ does it hold that the free extension of every effectively closed $N$-subshift to $N \times H$ is sofic? 
\end{question}

Finally, in~\Cref{thm:aperiodic} we proved that every product of two infinite and finitely generated groups with decidable word problem $H \times K$ admits a nonempty strongly aperiodic SFT as long as one of them is nonamenable. In~\cite{Barbieri_2019_DA} it was proven that the direct product of three infinite finitely generated groups with decidable word problem always admits a nonempty strongly aperiodic SFT, thus the remaining case is when we just have two groups and they are both amenable.

\begin{question}
    Let $G = H \times K$ be the direct product of two infinite finitely generated amenable groups with decidable word problem. Does $G$ admit a nonempty strongly aperiodic SFT?
\end{question}

Another interesting question is whether these strongly aperiodic SFTs can be constructed with properties of dynamical importance, such as mixing, transitivity, minimality, unique ergodicity, etc. Our current technique does not provide any such properties. This is particularly relevant due to the fact that the construction in~\cite{AuBiHTB_2022} for $G = \ZZ \times F_2$ is minimal.

\Addresses

\bibliographystyle{abbrv}
\bibliography{ref}

\end{document}